\documentclass[a4paper,12pt,leqno]{article}
\usepackage{amsmath}
\usepackage{amsthm}
\usepackage{amssymb}
\usepackage{amscd}
\usepackage{graphicx, color}
\usepackage{a4wide}
\usepackage[all]{xy}
\newtheorem{theorem}{Theorem}[section]
\newtheorem{lemma}[theorem]{Lemma}
\newtheorem{corollary}[theorem]{Corollary}

\theoremstyle{definition}
\newtheorem{definition}[theorem]{Definition}

\newtheorem{remark}[theorem]{Remark}

\numberwithin{equation}{section}

\usepackage[all]{xy}

\newcommand{\Z}{\Bbb Z}
\newcommand{\C}{\Bbb C}
\newcommand{\R}{\Bbb R}
\newcommand{\K}{\Bbb K}

\newcommand{\T}{\Bbb T}
\newcommand{\F}{\Bbb F}
\newcommand{\N}{\Bbb N}
\renewcommand{\P}{{\rm P}}
\newcommand{\SP}{\mbox{{\rm SP}}}

\newcommand{\RP}{\Bbb R\mbox{{\rm P}}}

\newcommand{\Map}{\mbox{{\rm Map}}}
\newcommand{\Rat}{\mbox{{\rm Rat}}}
\newcommand{\CP}{\Bbb C {\rm P}}

\newcommand{\dis}{\displaystyle}
\newcommand{\p}{\prime}

\newcommand{\Po}{\mbox{{\rm Poly}}}

\newcommand{\SZ}{{\mathcal{X}}^{d}}

\newcommand{\I}{\mbox{{\rm (i)}}}
\newcommand{\II}{\mbox{{\rm (ii)}}}
\newcommand{\III}{\mbox{{\rm (iii)}}}
\newcommand{\IV}{\mbox{{\rm (iv)}}}

\newcommand{\Q}{\mbox{{\rm Q}}}

\newcommand{\pol}{\mbox{{\rm Pol}}}
\title{\bf The homotopy type of spaces of real resultants
with bounded multiplicity}

\author{Andrzej Kozlowski\footnote{%
Institute of Applied Mathematics and Mechanics,
University of Warsaw, Banacha 2, 02-097 Warsaw, Poland
(E-mail: akoz@mimuw.edu.pl)
}
\  and \ 
Kohhei Yamaguchi\footnote{%
Department of Mathematics,
University of Electro-Communications,  Chofu, Tokyo 182-8585, Japan
(E-mail: kohhe@im.uec.ac.jp);
The second author is supported by 
JSPS KAKENHI Grant Number 26400083.
}}

\date{}

\begin{document}
\maketitle

\begin{abstract}
For positive integers $d,m,n\geq 1$ with
$(m,n)\not= (1,1)$
and $\K=\R$ or $\C$,
let $\Q^{d,m}_{n}(\K)$ denote the space of $m$-tuples
$(f_1(z),\cdots ,f_m(z))\in \K [z]^m$ of $\K$-coefficients monic polynomials of the same degree $d$ such that polynomials $\{f_k(z)\}_{k=1}^m$ have no common
{\it real} root of multiplicity $\geq n$
(but may have complex common root of any multiplicity).
These spaces can be regarded as one of generalizations
of the spaces defined and
studied by Arnold and
Vassiliev  \cite{Va}, and
they may be also considered as the
 {\it real} analogues of the spaces
studied by B. Farb and J. Wolfson \cite{FW}.
In this paper, we shall determine their homotopy types explicitly and
generalize the previous results obtained in \cite{Va} and \cite{KY1}.
\end{abstract}

\section{Introduction}\label{section 1}

\paragraph{Spaces of polynomials and the motivation.}
The principal motivation for this paper is derived from the 
two results obtained by Vassiliev \cite{Va} and Farb-Wolfson
\cite{FW}.
\par
Vassiliev \cite{Va} described a general method  for calculating
cohomology of certain spaces of polynomials (more precisely,
\lq\lq complements of discriminants\rq\rq ) by using a spectral sequence (to which and its variants we shall refer 
as the Vassiliev
 spectral sequence).
The most relevant example for us is the following.  For $\K=\R$ or $\C$,
let
$\P^d_n(\K)$  denote the space of
all $\K$-coefficients monic polynomials $f(z)\in \K [z]$ of degree $d$ which have
no {\it real} root of multiplicity $\geq n$
(but may have complex ones of arbitrary multiplicity).
By  identifying 
$S^1=\R\cup \{ \infty\}$ and $\C =\R^2$, we have
{\it the jet map}
\begin{equation}\label{eq: jet m=1}
j^{d,1}_{n,\K}:\P^d_n(\K) \to \Omega_{[d]_2}\RP^{d(\K)n-1}\simeq \Omega S^{d(\K)n-1}
\end{equation}
defined  by 
\begin{equation*}
j^{d,1}_{n,\K}(f(z))(\alpha )=
\begin{cases}
[f(\alpha ):f(\alpha )+f^{\p}(\alpha):
\cdots :f(\alpha )+f^{(n-1)}(\alpha)] & \mbox{if }\alpha\in\R
\\
[1:1:1:\cdots :1] &\mbox{if }\alpha =\infty
\end{cases}
\end{equation*}
for $(f(z),\alpha)\in\P^d_n(\K)\times S^1$, 
where $[d]_2\in \{0,1\}$ is the integer $d$ mod $2$, 
and
$d(\K)$ denotes the positive integer defined by
\begin{equation}
d(\K) = \dim_{\R}\K=
\begin{cases}
1 & \mbox{ if }\K =\R 
\\
2 & \mbox{ if }\K =\C 
\end{cases}
\end{equation}
For $\K =\R$, Vassiliev \cite{Va} obtained the following result:
\begin{theorem}[\cite{Va} (cf. \cite{GKY2}, \cite{KY1})]\label{thm: 1.1}
The jet map
$j^{d,1}_{n,\R}:\P^d_n(\R) \to \Omega_{[d]_2}\RP^{n-1}$ is a
homotopy equivalence through dimension $(\lfloor \frac{d}{n}\rfloor+1)(n-2)-1$ for $n\geq 4$ and a homology equivalence through dimension
$\lfloor \frac{d}{3}\rfloor$ for $n=3$, where
$\lfloor x\rfloor$ denotes the integer part of a real number $x$.
\qed
\end{theorem}

\begin{remark}
Let $X$ and $Y$ be based connected spaces.
Then
a based map $f:X\to Y$
is called {\it a homotopy equivalence}
(resp. {\it a homology equivalence})  
{\it through dimension} $N$
if the induced homomorphism
$
f_*:\pi_k(X)\to\pi_k(Y)
\quad
(\mbox{resp. }f_*:H_k(X;\Z) \to H_k(Y;\Z))
$
is an isomorphism for any integer $k\leq N$.
Similarly, when $G$ is a group and $f:X\to Y$ is a
$G$-equivariant map between $G$-spaces $X$ and $Y$, the map $f$ is called
a {\it $G$-equivariant homotopy equivalence through dimension} $N$
(resp. a {\it $G$-equivariant homology equivalence through dimension} $N$)
if the restriction map $f^H=f\vert X^H:X^H\to Y^H$ is a
homotopy  equivalence through dimension $N$ 
(resp. a homology equivalence through dimension $N$) for any subgroup
$H\subset G$, where $W^H$ denote the $H$-fixed subspace of a $G$-space $W$ given by
$W^H=\{x\in W: h\cdot x=x\mbox{ for any }h\in H\}.$
\qed
\end{remark}
\par
Next, recall the recent result obtained by Farb and Wolfson \cite{FW}. 
For positive integers $m,\ n\geq 1$ with $(m,n)\not= (1,1)$ and
a field $\F$ with its algebraic closure $\overline{\F}$,
let $\Po^{d,m}_n(\F)$ denote the space
of $m$-tuples $(f_1(z),\cdots ,f_m(z))\in \F [z]^m$
of $\F$-coefficients monic polynomials of the same degree $d$
such that the
polynomials $\{f_k(z)\}_{k=1}^m$ have no common root in
$\overline{\F}$ with multiplicity $\geq n$.
They studied the spaces 
$\Po^{d,m}_n(\F)$ from the point of view
of algebraic geometry
in the case when $\F=\mathbb{C}$ or $\F$ is a finite field $\F_q$. 
If $n=1$ and $\F =\C$, the space
$\Po^{d,m}_1(\C)$ can be identified with the space
$\Rat_d^*(\C \P^1,\C \P^{m-1})$  of all based rational maps
$f:\C \P^1\to \C \P^{m-1}$ of the degree $d$ (it is well known that in this case, rational maps coincide with holomorphic ones). 
The space of holomorphic maps appears in various applications and has been quite extensively studied (e.g. \cite{GKY4}, \cite{Se}).  
\par
In a different context, the same is true for the space $\Po^{d,1}_n(\C)$ 
(for $m=1$)
of monic  polynomials $f(z)\in \C [z]$ of the degree $d$ without $n$-fold roots. 
This space can be viewed as the space of polynomial functions \lq\lq without complicated singularities\rq\rq 
\ and thus plays an important role in singularity theory (\cite{Ar}, \cite{Va}). 
The fact that these two spaces were closely related was noted  and it was shown 
in \cite{Va} and \cite{CCMM} that they were stably homotopy equivalent.  
Recently this result was much improved  and  it was proved in \cite{KY8} 
(cf. \cite{GKY4}) that
there is a homotopy equivalence 
\begin{equation}\label{equ: KY8}
\Po^{d,m}_n(\F)
\simeq 
\Po^{\lfloor \frac{d}{n}\rfloor ,mn}_1(\F)
\quad \mbox{for }\F=\C \mbox{ and }
mn\geq 3.
\end{equation}
 
\par
Farb and Wolfson showed that some of these topological results have algebraic analogues when $\F$ is a finite field $\F_q$. This leads them to ask if these spaces in (\ref{equ: KY8}) are isomorphic as algebraic varieties over arbitrary fields $\F$. The affirmative answer would imply that the underlying topological spaces in the cases $\F=\Bbb C$ and $\F=\Bbb R$ are homeomorphic.\footnote{%
Recently H. Spink and D. Tseng showed that they are not isomorphic as varieties in
\cite{ST}.
}
Curtis McMullen indeed has shown that this is true in the simplest non-trivial case $(d,n)=(1,2)$
by constructing an explicit isomorphism.
Moreover, 
if $\F=\R$ and $n=1$, the space 
$\Po^{d,m}_1(\R)$
is precisely the space 
$\Rat_d^*(\F \P^1,\F \P^{m-1})$ of real rational functions considered by Segal in \cite{Se}. 
This space is an algebraic variety over $\Bbb R$ and one can ask if it is isomorphic (homeomorphic or homotopy equivalent?) to the variety
$\Po^{md,1}_m(\R)$ 
of real monic polynomials of degree $m d$ without $m$ fold complex roots. Of course, the affirmative answer to the 
Farb-Wolfson question would also imply this.
However, we have not been able to establish even homotopy equivalence in this case, 
and we will not consider this problem here. 
\par\vspace{2mm}\par
On the other hand, there is another space of real rational maps  
which can be viewed as the real analogue of the space of complex ones and it was first studied by Mostovoy 
in \cite{Mo1}. 
 Every such map can be represented by $n$-tuples of monic real polynomials of the same degree $d$ without a common real root (but possibly with common non-real roots).
However, in this situation the space of rational maps and the space of tuples of polynomials are 
different and we need to distinguish them.%
 \footnote{%
 For example,
 let $(f_1(z),f_2(z),f_3(z))\in \R [z]^3$ be a $3$-tuple of monic polynomials of the same degree $d$ without common real root.
 Then   the two $3$-tuples
 $F=((z^2+1)f_1(z),(z^2+1)f_2(z),(z^2+1)f_3(z))$
and  $G=((z^2+z+1)f_1(z),(z^2+z+1)f_2(z),(z^2+z+1)f_3(z))$ represent 
the same base-point preserving rational map
from $S^1$ to $\RP^2$, although $F\not= G$.
 Although it seems very likely to be true, 
 it has not been proved that the space of rational maps 
 and the space of tuples of polynomials are homotopy equivalent. 
}
 This space contains Segal's space of rational functions and is its closure in the space of  all continuous maps.  
 By analogy with the complex case, one can expect this space to be homotopy equivalent to the space of real 
 monic polynomials of degree $n d$ which do not have real roots of multiplicity $\geq n$. 
This is indeed true as was shown in \cite{KY1} (cf. \cite{Y5}). 
The spaces involved are not algebraic varieties; so the result is not implied by the positive answer to the Farb-Wolfson question. 
\par\vspace{1mm}\par
\par
The main purpose of this article is to generalize this result given in \cite{KY1}
for the space $\Q^{d,m}_n(\K)$, of $m$-tuples 
$(f_1(z),\cdots ,f_m(z))\in \K [z]^m$
of monic $\K$-coefficients polynomials of the same degree $d$, without $n$-fold common {\it real} roots (for $\K=\Bbb C$ of $\Bbb R$).
We will also prove that  an analogue of the homotopy equivalence
(\ref{equ: KY8}) 
 holds for the space $\Q^{d,m}_n(\K)$   
 (see Theorem \ref{thm: II} for the details). 

\par\vspace{1mm}\par

\paragraph{Basic definitions and notations.}

For connected spaces $X$ and $Y$, let
$\Map(X,Y)$ (resp. $\Map^*(X,Y)$) denote the space
consisting of all continuous maps
(resp. base-point preserving continuous maps) from $X$ to $Y$
with the compact-open topology, and
let $\RP^N$ (resp. $\CP^N$) denote the $N$-dimensional
real projective (resp. complex projective) space.
\par
Note that the based loop space 
$\Map^*(S^1,\RP^N)=\Omega \RP^N$
has two path-components $\Omega_{\epsilon}\RP^N$
for $\epsilon \in \{0,1\}$ when
$N\geq 2$.
The space $\Omega_0\RP^N$ is
the path-component of null homotopic maps and
$\Omega_1\RP^N$ is the path-component which contains
the natural inclusion of the bottom cell $S^1$ in
$\RP^N$.

\par
From now on, let $\K =\R$ or $\C$, 
let $d,m,n\geq 1$ be  positive integers 
such that $(m,n)\not= (1,1)$, and
we always assume that $z$ is a variable.
Let $\P^d(\K)$ denote the space of all $\K$-coefficients
monic polynomials 
$f(z)=z^d+a_1z^{d-1}+\cdots +a_d\in \K [z]$ of  degree $d$.

\begin{definition}\label{def: 1.3}
(i)
Let
$\Q^{d,m}_{n}(\K)$ denote the space consisting of
$m$-tuples $(f_1(z),\cdots ,f_m(z))\in \P^d(\K)^m$
of $\K$-coefficients monic polynomials of the same degree $d$ 
such that the
polynomials $f_1(z),\cdots ,f_m(z)$ have no common {\it real}
root of multiplicity $\geq n$
(but they may have a common {\it complex} root of any multiplicity).
\par
(ii)
Let $(f_1(z),\cdots ,f_m(z))\in \P^d(\K)^m$ be an
$m$-tuple of monic polynomials of the same degree $d$.
Then it is easy to see that 
$(f_1(z),\cdots ,f_m(z))\in \Q^{d,m}_{n}(\K)$ 
iff
the derivative polynomials
$\{f_j^{(k)}(z):
1\leq j\leq m,\ 0\leq k<n\}$
have no common real root.
Thus, by identifying $S^1=\R\cup\infty$,
one can define
 {\it the jet map}
\begin{equation}\label{equ: jet map}
j^{d,m}_{n,\K}:
\Q^{d,m}_{n}(\K)\to \Omega_{[d]_2} \RP^{d(\K)mn-1}
\simeq \Omega S^{d(\K)mn-1}
\quad
\mbox{by}
\end{equation}
\begin{equation}\label{eq: jet}
j^{d,m}_{n,\K}(f_1(z),\cdots ,f_m(z))(\alpha)
=
\begin{cases}
[\textit{\textbf{f}}_1(\alpha):\cdots :\textit{\textbf{f}}_m(\alpha)]
& \mbox{ if }\alpha \in\R
\\
[1:1:\cdots :1]
& \mbox{ if }\alpha =\infty
\end{cases}
\end{equation} 
for $(f_1(z),\cdots ,f_m(z))\in \Q^{d,m}_{n}(\K)$,
where we identify $\C =\R^2$ in 
(\ref{eq: jet}) if $\K=\C$, and
$\textit{\textbf{f}}_k(z)$ $(k=1,\cdots ,m)$
is the $n$-tuple of monic polynomials of the same degree $d$ defined by
\begin{equation}\label{equ: bff}
\textit{\textbf{f}}_k(z)=(f_k(z),f_k(z)+f^{\p}_k(z),f_k(z)+f^{\p\p}_k(z),
\cdots ,f_k(z)+f^{(n-1)}_k(z)).
\end{equation}
\end{definition}

\begin{remark}
Let
$f=(f_1(z),f_2(z),\cdots ,f_m(z))\in \Q^{d,m}_n(\K)$ be any element.
Then it is easy to see that
$\overline{f}=((z^2+1)f_1(z),(z^2+1)f_2(z),\cdots ,(z^2+1)f_m(z))\in \Q^{d+2,m}_n(\K).$
Since the two maps $j^{d,m}_{n,\K}(f)$ and $j^{d+2,m}_{n,\K}(\overline{f})$ 
are homotopic,
the image of $j^{d,m}_{n,\K}$ and that of
$j^{d+2,m}_{n,\K}$ are contained in the same path-component of 
$\Omega \RP^{d(\K)mn-1}$.
We can check directly that $j^{d,m}_{n,\K}(\Q^{d,m}_n(\K))
\subset \Omega_{[d]_2}\RP^{d(\K)mn-1}$ for $d\in \{1,2\}$.
Hence, the image $j^{d,m}_{n,\K}(\Q^{d,m}_n(\K))$
is contained in the component $\Omega_{[d]_2}\RP^{d(\K)mn-1}$ for any
$d\geq 1$.
\qed
\end{remark}

\begin{definition}

Note that
$\P^d_n(\K)=\Q^{d,1}_{n}(\K)$, and that
the map $j^{d,m}_{n,\K}$ coincides with the map
$j^{d,1}_{n,\K}$ given in
(\ref{eq: jet m=1}) for $m=1$.%
\footnote{%
Note that the space $\Q^{d,m}_n(\K)$ is also 
denoted by $\Q^d_{(m)}(\K)$ for $n=1$
in \cite{KY1}.
}
Similarly, 
one can define a natural map
\begin{equation}\label{eq: jet embedding}
i^{d,m}_{n,\K}:\Q^{d,m}_{n}(\K)\to \Q^{d,mn}_{1}(\K)
\qquad
\mbox{by}
\end{equation}
\begin{equation}
i^{d,m}_{n,\K}(f_1(z),\cdots ,f_m(z))=
\big(\textit{\textbf{f}}_1(z),\cdots ,\textit{\textbf{f}}_m(z)\big).
\end{equation}
\par
It is well-known that there is a homotopy equivalence (\cite{Ja}) 
\begin{equation}
\Omega S^{N+1}\simeq
S^N\cup e^{2N}\cup e^{3N}\cup \cdots \cup e^{kN}\cup
e^{(k+1)N}\cup \cdots .
\end{equation}
We will denote the $kN$-skeleton of $\Omega S^{N+1}$  by
$J_k(\Omega S^{N+1})$, i.e.
\begin{equation}
J_k(\Omega S^{N+1})\simeq
S^N\cup e^{2N}\cup e^{3N}\cup \cdots \cup e^{(k-1)N}\cup
e^{kN}.
\end{equation}
This space is usually called the
{\it $k$-stage James filtration} of $\Omega S^{N+1}$.
\qed
%
\end{definition}

\paragraph{Related known results. }
Let $D(d;m,n,\K)$ denote the positive integer defined by
\begin{align}\label{equ: number D}
D(d;m,n,\K)
&=(d(\K)mn-2)(\Big\lfloor \frac{d}{n}\Big\rfloor +1)-1
\\
&=
\begin{cases}
(2mn-2)(\lfloor \frac{d}{n}\rfloor +1)-1 &\mbox{if }\K=\C ,
\\
(mn-2)(\lfloor \frac{d}{n}\rfloor +1)-1 &\mbox{if }\K=\R .
\end{cases}
\nonumber
\end{align}
Recall the following known results for the case $m=1$ or $n=1$.

\begin{theorem}
[\cite{KY1}, \cite{Mo1}, \cite{Va}, \cite{Y5}]
\label{thm: KY1-I}
\par
$\I$
If $d(\K)m\geq 4$ and $n=1$,
the jet map 
$$
j^{d,m}_{1,\K}:\Q^{d,m}_{1}(\K)\to \Omega_{[d]_2}\RP^{d(\K)m-1}
\simeq \Omega S^{d(\K)m-1}
$$ 
is a
homotopy equivalence through dimension
$D(d;m,1,\K)$.
\par
$\II$
If $d(\K)n\geq 4$ and $m=1$,
the jet map 
$$
j^{d,1}_{n,\K}:\Q^{d,1}_n(\K)=\P^d_n(\K)\to \Omega_{[d]_2}\RP^{d(\K)n-1}
\simeq \Omega S^{d(\K)n-1}
$$ 
is a
homotopy equivalence through dimension
$D(d;1,n,\K)$.
\par
$\III$
If 
 $d(\K)n\geq 4$, 
there are  homotopy equivalences
$$
\Q^{d,1}_n(\K)=\P^d_n(\K) \simeq  
J_{\lfloor\frac{d}{n}\rfloor}(\Omega S^{d(\K)n-1})
\ \mbox{ and }\ 
\Q^{d,n}_1(\K)\simeq
J_d(\Omega S^{d(\K)n-1}).
$$
Thus, there is a homotopy equivalence
$
\Q^{d,1}_n(\K)=\P^d_n(\K) \simeq \Q^{\lfloor \frac{d}{n}\rfloor,n}_{1}(\K)
$
if $d(\K)n\geq 4$.
\par
$\IV$
In particular,
if $(\K,m)=(\R,3)$ and $d\geq 1$ is an odd integer,
there is a homotopy equivalence
$\Q^{d,3}_1(\R)\simeq J_d(\Omega S^2)$.
\qed
\end{theorem}

Note that
the conjugation on $\C$ naturally induces a $\Z/2$-action on
the space $\Q^{d,m}_{n}(\C)$.
From now on,
we regard $\RP^N$ as the $\Z/2$-space with trivial
$\Z/2$-action, and recall
the following result given in \cite{KY1}.

\begin{theorem}[\cite{KY1}]\label{crl: KY1-I}
$\I$
If $m\geq 4$, then
the jet map 
$$
j^{d,m}_{1,\C}:\Q_{1}^{d,m}(\C)\to \Omega_{[d]_2}\RP^{2m-1}
\simeq \Omega S^{2m-1}
$$ 
is a $\Z/2$-equivariant 
homotopy equivalence through dimension
$D(d;m,1,\R)$.
\par
$\II$
If $n\geq 4$, then
the jet map 
$$
j^{d,1}_{n,\C}:\Q^{d,1}_n(\C)=\P^d_n(\C)\to \Omega_{[d]_2}\RP^{2n-1}
\simeq \Omega S^{2n-1}
$$ 
is a $\Z/2$-equivariant 
homotopy equivalence through dimension
$D(d;1,n,\R)$.
\qed
\end{theorem}
\paragraph{The main results. }
The main purpose of this paper is to determine the homotopy type of 
the space $\Q^{d,m}_{n}(\K)$ explicitly and 
generalize the above two theorems (Theorems
\ref{thm: KY1-I}
and \ref{crl: KY1-I}) for the case $m\geq 2$ and the case $n\geq 2$.
More precisely, the main results are below.

\begin{theorem}\label{thm: I}
If $d(\K)mn\geq 4$,
the jet map
$$
j^{d,m}_{n,\K}:\Q^{d,m}_{n}(\K) \to \Omega_{[d]_2}\RP^{d(\K)mn-1}\simeq \Omega S^{d(\K)mn-1}
$$
is a homotopy equivalence through dimension
$D(d;m,n,\K)$.
\end{theorem}
Note that the conjugation on $\C$ naturally induces the $\Z/2$-action on
the space $\Q^{d,m}_{n}(\C)$. 
Since the map $j^{d,m}_{n,\C}$ is a
$\Z/2$-equivariant map and $(j^{d,m}_{n,\C})^{\Z/2}=j^{d,m}_{n,\R}$,
we also obtain the following result.

\begin{corollary}\label{cor: I-1}
If $mn\geq 4$, the jet map
$$
j^{d,m}_{n,\C}:\Q^{d,m}_{n}(\C) \to \Omega_{[d]_2}\RP^{2mn-1}\simeq 
\Omega S^{2mn-1}
$$
is a $\Z/2$-equivariant homotopy equivalence through dimension
$D(d;m,n,\R)$.
\qed
\end{corollary}

\begin{corollary}\label{cor: I-2}
If $d(\K)mn\geq 4$, the jet embedding
\begin{equation*}
i^{d,m}_{n,\K}:\Q^{d,m}_{n}(\K)
\to
\Q^{d,mn}_1(\K)
\end{equation*}
is a homotopy equivalence through dimension $D(d;m,n,\K)$.
\end{corollary}

Finally
we have the following result.

\begin{theorem}\label{thm: II}
If $d(\K)mn\geq 4$,
there is a homotopy equivalence
$$
\Q^{d,m}_{n}(\K) \simeq 
J_{\lfloor \frac{d}{n}\rfloor}(\Omega S^{d(\K)mn-1}).
$$
Hence, in this situation, there are homotopy equivalences
$$
\Q^{d,m}_{n}(\K)\simeq \Q^{d,1}_{mn}(\K)
\simeq
\Q^{\lfloor\frac{d}{n}\rfloor ,mn}_{1}(\K). 
$$
\end{theorem}

This paper is organized as follows.
In \S \ref{section: simplicial resolution} we recall the simplicial
resolutions and in \S \ref{section: spectral sequence}
we construct the Vassiliev spectral sequences induced from the non-degenerate 
simplicial resolutions.
In particular, by using this spectral sequence, we compute
the homology 
$H_*(\Q^{d,m}_{n}(\K),\Z)$ explicitly.
In \S \ref{section: sd}, we recall the stabilization maps
 and prove the key unstable
result (Theorem \ref{thm: stab1})
by using comparison of the Vassiliev spectral sequence.
In \S \ref{section: scanning maps} we prove the stability result 
(Theorem \ref{thm: natural map}) by using the horizontal scanning map.
Finally in \S \ref{section: proofs} we give the proof of the main results
(Theorem \ref{thm: I},
Corollary \ref{cor: I-2}, Theorem \ref{thm: II})
by using two key results (Theorem \ref{thm: stab1}, Theorem \ref{thm: natural map}).

\section{Simplicial resolutions}\label{section: simplicial resolution}

In this section, we give the definitions of and summarize the basic facts about  non-degenerate simplicial resolutions  
(\cite{Va}, \cite{Va2}, (cf.  \cite{Mo2})).
\begin{definition}\label{def: def}
{\rm
(i) For a finite set $\textbf{\textit{v}} =\{v_1,\cdots ,v_l\}\subset \R^N$,
let $\sigma (\textbf{\textit{v}})$ denote the convex hull spanned by 
$\textbf{\textit{v}}.$
Let $h:X\to Y$ be a surjective map such that
$h^{-1}(y)$ is a finite set for any $y\in Y$, and let
$i:X\to \R^N$ be an embedding.
Let  $\mathcal{X}^{\Delta}$  and $h^{\Delta}:{\mathcal{X}}^{\Delta}\to Y$ 
denote the space and the map
defined by
\begin{equation}\label{eq: non-degenerate}
\mathcal{X}^{\Delta}=
\big\{(y,u)\in Y\times \R^N:
u\in \sigma (i(h^{-1}(y)))
\big\}\subset Y\times \R^N,
\ h^{\Delta}(y,u)=y.
\end{equation}
The pair $(\mathcal{X}^{\Delta},h^{\Delta})$ is called
{\it the simplicial resolution of }$(h,i)$.
In particular, $(\mathcal{X}^{\Delta},h^{\Delta})$
is called {\it a non-degenerate simplicial resolution} if for each $y\in Y$
any $k$ points of $i(h^{-1}(y))$ span a $(k-1)$-dimensional simplex of $\R^N$.
\par
(ii)
For each $k\geq 0$, let $\mathcal{X}^{\Delta}_k\subset \mathcal{X}^{\Delta}$ be the subspace
given by 
\begin{equation}\label{eq: simplicial}
\mathcal{X}_k^{\Delta}=\big\{(y,u)\in \mathcal{X}^{\Delta}:
u \in\sigma (\textbf{\textit{v}}),
\textbf{\textit{v}}=\{v_1,\cdots ,v_l\}\subset i(h^{-1}(y)),\ l\leq k\big\}.
\end{equation}
We make identification $X=\mathcal{X}^{\Delta}_1$ by identifying 
 $x\in X$ with 
$(h(x),i(x))\in \mathcal{X}^{\Delta}_1$,
and we note that  there is an increasing filtration
\begin{equation*}\label{equ: filtration}
\emptyset =
\mathcal{X}^{\Delta}_0\subset X=\mathcal{X}^{\Delta}_1\subset \mathcal{X}^{\Delta}_2\subset
\cdots \subset \mathcal{X}^{\Delta}_k\subset \mathcal{X}^{\Delta}_{k+1}\subset
\cdots \subset \bigcup_{k= 0}^{\infty}\mathcal{X}^{\Delta}_k=\mathcal{X}^{\Delta}.
\end{equation*}
}
\end{definition}
Since the map $h^{\Delta}$ is a proper map,
it extends to the map
$h^{\Delta}_+:\mathcal{X}^{\Delta}_+\to Y_+$
between one-point compactifications, where
$X_+$ denotes the one-point compactification of a locally compact space
$X$.

\begin{theorem}[\cite{Va}, \cite{Va2} , \cite{Mo2}
(cf. \cite{KY7})]\label{thm: simp}
Let $h:X\to Y$ be a surjective map such that
$h^{-1}(y)$ is a finite set for any $y\in Y,$ 
$i:X\to \R^N$ an embedding, and let
$(\mathcal{X}^{\Delta},h^{\Delta})$ denote the simplicial resolution of $(h,i)$.
\par
\begin{enumerate}
\item[$\I$]
If $X$ and $Y$ are semi-algebraic spaces and the
two maps $h$, $i$ are semi-algebraic maps, then
$h^{\Delta}_+:\mathcal{X}^{\Delta}_+\stackrel{\simeq}{\rightarrow}Y_+$
is a homology equivalence.
\item[$\II$]
If there is an embedding $j:X\to \R^M$ such that its associated simplicial resolution
$(\tilde{\mathcal{X}}^{\Delta},\tilde{h}^{\Delta})$
is non-degenerate,
the space $\tilde{\mathcal{X}}^{\Delta}$
is uniquely determined up to homeomorphism and
there is a filtration preserving homotopy equivalence
$q^{\Delta}:\tilde{\mathcal{X}}^{\Delta}\stackrel{\simeq}{\rightarrow}{\mathcal{X}}^{\Delta}$ such that $q^{\Delta}\vert X=\mbox{id}_X$.
\item[$\III$]
A non-degenerate simplicial resolution always exists even if the map $h$ is not finite to one.
\end{enumerate}
\end{theorem}
\begin{proof}\label{Remark: non-degenerate}
The assertion  (ii) follows from the universality of non-degenerate  resolutions as in \cite[page 287]{Mo2}, and it remains to show (i) and (iii).
Note that  (i)  was 
already  proved in \cite[Lemma 1 (page 90)]{Va}.
Indeed,
by \cite[Theorem (page 43)]{GM}
semi-algebraic sets can be triangulated; therefore,   there exists a cellular decomposition
of the space $Y_+$ such that  over each open cell the projection 
$h^{\Delta}_+$ is a trivializable bundle with simplex as a fiber.
Filter the space $Y_+$ by the skeletons of this decomposition, and the space
$\mathcal{X}^{\Delta}$ by the pre-images of these skeletons under the map
$h^{\Delta}_+$.
Let $\mathcal{E}^r_{p,q}$ denote the spectral sequence 
induced from the this filtration
converging to the homology of
the space $\mathcal{X}^{\Delta}_+$.
By construction of this filtration, we easily see that
$\mathcal{E}^1_{p,q}=0$ for any $q\geq 1$, and that
the complex $\{\mathcal{E}^1_{*,0};d_1\}$ 
is isomorphic to the cellular differential complex of the
cellular decomposition of $Y_+$, the isomorphism being given by  
$(h^{\Delta}_+)_{\#}$.
Thus, $h^{\Delta}_+$ is a homology equivalence and (i) was obtained.
\par
Finally, we prove (iii).
Recall from
 \cite[Chap. III]{Va}  that there exists a sequence of embeddings
$\{\tilde{i}_k:X\to \R^{N_k}\}_{k\geq 1}$ satisfying the following two conditions
for each $k\geq 1$:
\begin{enumerate}
\item[(\ref{eq: simplicial}.i)]
For any $y\in Y$,
any $t$ points of the set $\tilde{i}_k(h^{-1}(y))$ span a $(t-1)$-dimensional affine subspace
of $\R^{N_k}$ if $t\leq 2k$.
\item[(\ref{eq: simplicial}.ii)]
$N_k\leq N_{k+1}$ and if we identify $\R^{N_k}$ with a subspace of
$\R^{N_{k+1}}$, 
then $\tilde{i}_{k+1}=\hat{i}\circ \tilde{i}_k$,
where
$\hat{i}:\R^{N_k}\stackrel{\subset}{\rightarrow} \R^{N_{k+1}}$
denotes the inclusion.
\end{enumerate}
We then let
\begin{equation}\label{2.3}
\mathcal{X}^{\Delta}_k=\big\{(y,u)\in Y\times \R^{N_k}:
u\in\sigma (\textbf{\textit{v}}),
\textbf{\textit{v}}
=\{v_1,\cdots ,v_l\}\subset \tilde{i}_k(h^{-1}(y)),l\leq k\big\}.
\end{equation}
Identifying naturally  $\mathcal{X}^{\Delta}_k$ with a subspace
of $\mathcal{X}_{k+1}^{\Delta}$ by (\ref{thm: simp}.ii),  we obtain the non-degenerate simplicial
resolution $\mathcal{X}^{\Delta}$ of  $h$ as the union  
$\dis \mathcal{X}^{\Delta}=\bigcup_{k\geq 1} \mathcal{X}^{\Delta}_k$.
\end{proof}

\begin{remark}\label{Remark: simp}
 It is known that $h^{\Delta}_+$ is  actually a homotopy equivalence 
\cite[page 156]{Va2}. 
However, in this paper we do not need this stronger assertion.
\qed
\end{remark}
\begin{definition}
Let $\tilde{i}_k:X\to \R^{N_k}$ be an embedding satisfying the condition
(\ref{eq: simplicial}.i).
For a finite set $\{u_i\}_{i=1}^k\subset X$,
let $\Delta (u_1,\cdots ,
u_k)$ denote the convex hull spanned by the vectors
$\{\tilde{i}_k(u_1),\cdots ,\tilde{i}_k(u_k)\}$
and we denote by
$\mbox{Int}(\Delta (u_1,\cdots ,u_k))$ the set of interior points of
$\Delta (u_1,\cdots ,u_k)$.
\qed
\end{definition}
\begin{lemma}\label{lmm: simplicial}
Suppose that
$\{x_i\}_{i=1}^k\subset \tilde{i}_k((h^{\Delta})^{-1}(y))$ and
$\{y_i\}_{i=1}^k\subset \tilde{i}_k((h^{\Delta})^{-1}(y))$
for some $y\in Y$.
If $\{x_i\}_{i=1}^k\not=\{y_i\}_{i=1}^k$,
$\mbox{\rm Int}(\Delta (x_1,\cdots ,x_k))\cap \mbox{\rm Int}(\Delta (y_1,\cdots ,y_k))=\emptyset.$
\end{lemma}
\begin{proof}
Suppose that 
$\mbox{\rm Int}(\Delta (x_1,\cdots ,x_k))\cap \mbox{\rm Int}(\Delta (y_1,\cdots ,y_k))\not= \emptyset.$ Then there is an element
$\alpha \in \mbox{Int}(\Delta (x_1,\cdots ,x_k))\cap \mbox{Int}(\Delta (y_1,\cdots ,y_k))$, and
one can write
$\alpha =\sum_{i=1}^k\lambda_i\tilde{i}_k(x_i)=\sum_{i=1}^k\mu_i\tilde{i}_k(y_i)$
for  $\lambda_i,\mu_i>0$
$(i=1,\cdots ,k)$ with $\sum_{i=1}^k\lambda_i=\sum_{i=1}^k\mu_i=1$.
Since $\{x_i\}_{i=1}^k\not=\{y_i\}_{i=1}^k$, by reindexing  we may suppose that there is a positive integer $1\leq t\leq k$ such that
$\{x_i\}_{i=1}^k\cup\{y_i\}_{i=1}^k=\{x_1,\cdots ,x_t,y_1,\cdots ,y_t,z_{t+1},\cdots ,z_k\}$
satisfying the the conditions $\{x_i\}_{i=1}^t\cap \{y_i\}_{i=1}^t=\emptyset$
and $x_i=y_i$ for any $t+1\leq i\leq k.$
Since $2t+(k-t)=k+t\leqq 2k$, 
by using the condition (\ref{eq: simplicial}.i) we see that the elements
$\{\tilde{i}_k(x_i),\tilde{i}_k(y_i),\tilde{i}_k(z_j):1\leq i\leq t< j\leq k\}$ are
affinely independent.
Since $\sum_{i=1}^k\lambda_i\tilde{i}_k(x_i)-\sum_{i=1}^k\mu_i\tilde{i}_k(y_i)
=\sum_{i=1}^t\lambda_i\tilde{i}_k(x_i)-\sum_{i=1}^t\mu_i\tilde{i}_k(y_i)
+\sum_{i=t+1}^k(\lambda_i-\mu_i)z_i={\bf 0}$
and
$\sum_{i=1}^t\lambda_i-\sum_{i=1}^t\mu_i+\sum_{i=t+1}^k(\lambda_i-\mu_i)=0$,
 we see that $\lambda_i=\mu_i=0$ for any $1\leq i\leq t$. 
 But this is a contradiction.
 This completes the proof.
\end{proof}

\section{The Vassiliev spectral sequence}
\label{section: spectral sequence}

In this section we construct a spectral sequence similar to the one frequently used by
Vassiliev  in \cite{Va} (which will will  call simply  \lq\lq the Vassiliev spectral sequence\rq\rq) converging to the homology of
$\Q^{d,m}_{n}(\K)$ and compute it explicitly.

\begin{definition}\label{Def: 3.1}
\par
(i)
Let $\K =\R$ or $\C$, and let
$\Sigma^{d,m}_{n,\K}$ denote \emph{the discriminant} of
$\Q^{d,m}_{n}(\K)$ in $\P^d(\K)^m$ given by
the complement
\begin{eqnarray*}
 \Sigma^{d,m}_{n,\K} &=&
\P^d(\K)^m\setminus \Q^{d,m}_{n}(\K)
\\
&=&
\big\{(f_1,\cdots ,f_{m})\in \P^d(\K)^m :
\textbf{\textit{f}}_1(x)=\cdots
=\textbf{\textit{f}}_m(x)=\mathbf{0}
\mbox{ for some }x\in \R\big\}.
\end{eqnarray*}
\par
(ii)
Let
$Z^{d,m}_{n,\K}\subset \Sigma^{d,m}_{n,\K}\times\R$ denote 
{\it the tautological normalization} of $\Sigma^{d,m}_{n,\K}$ 
given by
$$
Z^{d,m}_{n,\K}=\big\{
((f_1(z),\cdots ,f_m(z),x)\in \Sigma^{d,m}_{n,\K}\times\R:
\textit{\textbf{f}}_1(x)=\cdots =\textit{\textbf{f}}_m(x)=\mathbf{0}
\big\}.
$$
Let
$\pi^{d,m}_{n,\K}:Z^{d,m}_{n,\K}\to \Sigma^{d,m}_{n,\K}$
denote the map given by the projection to the first factor, and
let
$(\SZ,{\pi}^{\Delta}:\SZ (\K)\to\Sigma^{d,m}_{n,\K})$ 
be the  non-degenerate simplicial resolutions of $\pi^{d,m}_{n,\K}$
as in Theorem \ref{thm: simp}.
Note that
there is a
natural increasing filtration
\begin{eqnarray*}
\emptyset 
&=&
\SZ_0
\subset \SZ_1\subset 
\SZ_2\subset \cdots \cdots\subset
\bigcup_{k= 0}^{\infty}\SZ_k=\SZ (\K).
\end{eqnarray*}
Since any $(f_1(z),\cdots ,f_m(z))\in \Sigma^{d,m}_{n,\K}$
has at most $\lfloor \frac{d}{n}\rfloor$ distinct common real roots
of multiplicity $n$,
the following equality holds:
\begin{equation}\label{filt}
\mathcal{X}^d_k=\mathcal{X}^d(\K) 
\quad
\mbox{ if }k\geq \Big\lfloor \frac{d}{n}\Big\rfloor .
\end{equation}
\end{definition}
By Theorem \ref{thm: simp},
the map
$\pi_{+}^{\Delta}:\SZ (\K)_+\stackrel{\simeq}{\longrightarrow}
(\Sigma^{d,m}_{n,\K})_+$
is a homology equivalence.
Since
${\mathcal{X}_k^{d}}_+/{\SZ_{k-1}}_+
\cong 
(\SZ_k\setminus \SZ_{k-1})_+$,
we have a spectral sequence 
$$
\big\{E_{t;d}^{k,s},
d_t:E_{t;d}^{k,s}\to E_{t;d}^{k+t,s+1-t}
\big\}
\Rightarrow
H^{k+s}_c(\Sigma^{d,m}_{n,\K};\Z),
$$
such that
$E_{1;d}^{k,s}
=H^{k+s}_c(\SZ_k\setminus\SZ_{k-1};\Z)$,
where
$H_c^k(X;\Z)$ denotes the cohomology group with compact supports given by 
$
H_c^k(X;\Z)= H^k(X_+;\Z).
$
\par
Since there is a homeomorphism
$\P^d(\K)^m\cong \R^{d(\K)md}$,
by Alexander duality  there is a natural
isomorphism
\begin{equation}\label{Al}
\tilde{H}_k(\Q^{d,m}_{n}(\K);\Z)\cong
H_c^{d(\K)md-k-1}(\Sigma_{n,\K}^{d,m};\Z)
\quad
\mbox{for any }k.
\end{equation}
By
reindexing we obtain a
spectral sequence
\begin{eqnarray}\label{SS}
&&\big\{E^{t;d}_{k,s}, d^{t}:E^{t;d}_{k,s}\to E^{t;d}_{k+t,s+t-1}
\big\}
\Rightarrow \tilde{H}_{s-k}(\Q^{d,m}_{n}(\K);\Z),
\end{eqnarray}
where
$E^{1;d}_{k,s}=
H^{d(\K)md+k-s-1}_c(\SZ_k\setminus\SZ_{k-1};\Z).$

For a space $X$, let $F(X,k)\subset X^k$ denote the ordered
configuration space of distinct $k$ points in $X$ given by
\begin{equation}
F(X,k)=\{(x_1,\cdots ,x_k)\in X^k:
x_i\not= x_j\mbox{ if }i\not= j\}.
\end{equation}
Let $S_k$ be the symmetric group on $k$ letters.
Then the group $S_k$ acts on $F(X,k)$ by permuting coordinates
and we let $C_k(X)$ denote the orbit space
\begin{equation}
C_k(X):=F(X,k)/S_k.
\end{equation}

\begin{lemma}\label{lemma: vector bundle}
If  
$1\leq k\leq \lfloor \frac{d}{n}\rfloor$,
$\SZ_k\setminus\SZ_{k-1}$
is homeomorphic to the total space of a real affine
bundle $\xi_{d,k}$ over $C_k(\R)$ of rank 
\begin{equation}
l_{d,k}:=d(\K)m(d-nk)+k-1.
\end{equation}
\end{lemma}
\begin{proof}

The argument is exactly analogous to the one in the proof of  
\cite[Lemma 4.4]{AKY1}. 
Namely, an element of $\SZ_k\setminus\SZ_{k-1}$ is represented by
an $(m+1)$-tuple 
$(f_1(z),\cdots ,f_m(z),u)$, where 
$(f_1(z),\cdots ,f_m(z))$ is an $m$-tuple of monic polynomials of the same
degree $d$
in $\Sigma^{d,m}_{n,\K}$ and $u$ is an element of the interior of
the span of the images of $k$ distinct points 
$\{x_i\}_{i=1}^k\in C_k(\R)$ 
such that
$\{x_j\}_{j=1}^k$ are common roots of 
$\{f_i(z)\}_{k=1}^m$ of multiplicity $n$
under a suitable embedding $\tilde{i}_k$ satisfying the condition (\ref{eq: simplicial}.i).
\ 
Note that the $k$ distinct points $\{x_j\}_{j=1}^k$ 
are uniquely determined by $u$.
Indeed, if there exists another set of common roots  $\{y_i\}_{i=1}^k\in C_k(\R)$ 
of $\{f_i(z)\}_{i=1}^m$ of multiplicity $n$ satisfying the same condition,
then $u\in \mbox{Int}(\Delta (x_1,\cdots ,x_k))\cap
\mbox{Int}(\Delta (y_1,\cdots ,y_k))$.
However, if $\{x_i\}_{i=1}^k\not=\{y_i\}_{i=1}^k$, then,
by Lemma \ref{lmm: simplicial},
$\mbox{Int}(\Delta (x_1,\cdots ,x_k))\cap
\mbox{Int}(\Delta (y_1,\cdots ,y_k))=\emptyset$ and this is a contradiction.
Thus, we have a projection map
$\pi_{k,d} :{\cal X}^{d}_k\setminus
{\cal X}^{d}_{k-1}\to C_{k}(\R)$
defined by
$((f_1,\cdots ,f_m),u) \mapsto 
\{x_1,\cdots ,x_k\}$. 
\par
Now suppose that $1\leq k\leq \lfloor \frac{d}{n}\rfloor$
and $1\leq i\leq m$.
Let $c=\{x_j\}_{j=1}^k\in C_{k}(\R)$
 be any fixed element and consider the fibre  $\pi_{k,d}^{-1}(c)$.
It is easy to see that the condition
for a polynomial $f_i(z)\in\P^d(\K)$
to be divisible by
the polynomial
$\prod_{j=1}^k(z-x_j)^n$,
is equivalent
to the following:
\begin{equation}\label{equ: equation}
f^{(t)}_i(x_j)=0
\quad
\mbox{for }0\leq t<n,\ 1\leq j\leq k.
\end{equation}
In general, for each $0\leq t< n$ and $1\leq j<k$,
the condition $f^{(t)}_i(x_j)=0$ 
gives
one  linear condition on the coefficients of $f_i(z)$,
and determines an affine hyperplane in $\P^d(\K)$. 
For example, if we set $f_i(z)=z^d+\sum_{s=1}^da_{s}z^{d-s}$,
then
$f_i(x_j)=0$ for all $1\leq j\leq k$
if and only if
$A_1\textbf{\textit{x}}=\textbf{\textit{b}}_1$, where we set
\begin{equation*}\label{equ: matrix equation}
A_1=
\begin{bmatrix}
1 \ & x_1 \ & x_1^2 \ & x_1^3\  & \cdots \ &  x_1^{d-1} \ 
\\
1 \ & x_2 \ & x_2^2\ & x_2^3 \ & \cdots \ & x_2^{d-1} \ 
\\
\vdots & \ddots & \ddots & \ddots & \ddots & \vdots
\\
1 \ & x_k \ & x_k^2 \ & x_k^3 \ & \cdots  \ & x_k^{d-1} \ 
\end{bmatrix}
,\quad
\textbf{\textit{x}}=
\begin{bmatrix}
a_{d}\\ a_{d-1} \\ \vdots 
\\ a_{1}
\end{bmatrix}
,\quad
\textbf{\textit{b}}_1=
-
\begin{bmatrix}
x_1^d\\ x_2^d \\ \vdots 
\\ x_k^d
\end{bmatrix}.
\end{equation*}
Similarly, $f^{\p}_i(x_j)=0$ for all $1\leq j\leq k$
if and only if
$A_2\textbf{\textit{x}}=\textbf{\textit{b}}_2$, where we set
\begin{equation*}\label{equ: matrix equation2}
A_2=
\begin{bmatrix}
0 \ &1 \ & 2x_1\  & 3x_1^2 \ & \cdots \ & (d-1)x_1^{d-2} \ 
\\
0 \ & 1 \ & 2x_2 \ & 3x_2^2\ & \cdots \  & (d-1)x_2^{d-2} \ 
\\
\vdots & \vdots & \ddots & \ddots & \ddots & \vdots
\\
0 \ &1 \ & 2x_k \ & 3x_k^2 \ & \cdots \ & (d-1)x_k^{d-2} \ 
\end{bmatrix}
,\quad
\textbf{\textit{b}}_2=
-
\begin{bmatrix}
dx_1^{d-1}\\ dx_2^{d-1} \\ \vdots 
\\ dx_k^{d-1}
\end{bmatrix}.
\end{equation*}
Analogously,
$f^{\p\p}_i(x_j)=0$ for all $1\leq j\leq k$
if and only if $A_3\textbf{\textit{x}}=\textbf{\textit{b}}_3$, where we set
\begin{equation*}\label{equ: matrix equation2}
A_3=
\begin{bmatrix}
0 \ & 0 & 2 \ & 6x_1 \ & \cdots \ & (d-1)(d-2)x_1^{d-3} \ 
\\
0 \ & 0 \  & 2 \ & 6x_2 \ & \cdots \ & (d-1)(d-2)x_2^{d-3} \ 
\\
\vdots & \vdots & \ddots & \ddots & \ddots & \vdots
\\
0 \ & 0 \ & 2 \ & 6x_k \ & \cdots \ & (d-1)(d-2)x_k^{d-3} \ 
\end{bmatrix},
\quad
\textbf{\textit{b}}_3=
-
\begin{bmatrix}
d(d-1)x_1^{d-2}\\ 
d(d-1)x_2^{d-1} 
\\ \vdots 
\\ d(d-1)x_k^{d-2}
\end{bmatrix}
,
\end{equation*}
and so on.
Since $\{x_i\}_{i=1}^k\in C_k(\R)$, 
by Gaussian elimination of rows of
matrices, 
the matrix $A_1$ reduces to the matrix $B_1$, where
$s_i(t):=\sum_{i_1+\cdots +i_t=i}x_1^{i_1}x_2^{i_2}\cdots x_t^{i_t}$
and
$$
B_1=
\begin{bmatrix}
1 \ & x_1 \ & x_1^2\  & x_1^3 \ & x_1^4 \  & x_1^5 \  & \cdots \ &\cdots \ & x_1^{d-2} & x_1^{d-1} \ 
\\
0 \ & 1 \ & s_1(2)\  & s_2(2)\  & s_3(2)\  & s_4(2)\  & \cdots \ & \cdots \ & s_{d-3}(2)\  & s_{d-2}(2)
\\
0 \ & 0 \ & 1 \ & s_1(3) \ & s_2(3) \  & s_3(3) \ & \cdots & \cdots & s_{d-4}(3)\  & s_{d-3}(3) \ 
\\
0 & 0 & 0 & 1 & s_1(4) & s_2(4) & \cdots & \cdots & s_{d-5}(4) & s_{d-4}(4)
\\
\vdots  &\vdots & \ddots &\ddots  &\ddots & \ddots &\ddots  &\ddots &\vdots   &\vdots
\\
0 & \cdots &\cdots  & \cdots & 0 & 1 & s_1(k)\mbox{ } &s_2(k) & \cdots & s_{d-k}(k)
\end{bmatrix}
$$
Similarly, 
by easy Gaussian elimination of rows of
matrices, 
the matrix $A_2$ 
reduces to the matrix $B_2$,
where
{\small
\begin{align*}
B_2&=
\begin{bmatrix}
 0 & 1 & 2x_1 & 3x_1^2 & 4x_1^3 & 5x_1^4 & \cdots & \cdots & \cdots &   (d-1)x_1^{d-2}
 \\
0 & 0 & 2 &3s_1(2)\  & 4s_2(2)\  & 5s_3(2)\  & \cdots & \cdots &\cdots &
  (d-1)s_{d-3}(2)
 \\
 0 & 0 & 0 & 3 & 4s_1(3)\  & 5s_2(3)\  & \cdots & \cdots & \cdots &
 (d-1)s_{d-4}(3)
 \\
 \vdots & \vdots &\ddots & \ddots & \ddots & \ddots & \ddots &\ddots & \ddots & \vdots
 \\
0 & \cdots & \cdots &\cdots & 0 & k-1\mbox{ } &ks_1(k)\mbox{ }&
(k+1)s_2(k)\mbox{ } & \cdots &
(d-1)s_{d-k-1}(k)
\end{bmatrix}.
\end{align*}
}
\par
Analogously, the matrix $A_3$ 
reduces to the matrix $B_3$,
where
{\small
\begin{align*}
B_3&=
\begin{bmatrix}
0 & 0 & 2 & 6x_1 & 12 x_1^2 & 20x_1^3 \cdots &\cdots &
\cdots &
d(d-1)x_1^{d-3}
\\
0 & 0& 0 &6 & 12 s_1(2)\  &20 s_2(2)\  &&&
d(d-1)s_{d-k-2}(2)
\\
\vdots& \vdots &\ddots&\ddots &\ddots &\ddots & \ddots & \ddots &
\vdots
\\
0& \cdots &\cdots &\cdots & 0 & (k-2)(k-3)\mbox{ }\mbox{ }  &k(k-1)s_1(k)
\mbox{ }&\cdots
&\  d(d-1)s_{d-k-3}(k)
\end{bmatrix}.
\end{align*}
}
\par
If we repeat this process, we finally obtain the reduced $(k\times d)$
matrices $\{B_t\}_{t=1}^n$
such that each $A_t$ reduces to the matrix $B_t$.
Now 
define the $(tk\times d)$ matrix $C_t$ (for $1\leq t\leq n)$
inductively
by $C_1=B_1$ and
$C_t=
\begin{bmatrix}
C_{t-1}
\\
B_t
\end{bmatrix}$
for $2\leq t\leq n$.
Then
by the induction $t$ and some tedious calculations, we see that
each matrix $C_t$ has rank $kt$ for each $1\leq t\leq n$.
Thus we see  
that the the condition  
(\ref{equ: equation}) 
gives exactly $nk$ affinely independent conditions on the coefficients of $f_i(z)$.
Hence we see that
the space of $m$-tuples $(f_1(z),\cdots ,f_m(z))\in\P^d(\K)^m$ 
of monic polynomials which satisfy
the condition (\ref{equ: equation}) for each $1\leq i\leq m$
is the non-trivially and transversely intersection of $mnk$ affine hyperplanes, and
it has codimension $mnk$ in $\P^d(\K)^m$.
Therefore,
the fibre $\pi_{k,d}^{-1}(c)$ is homeomorphic  
to the product of an open $(k-1)$-simplex
 with the real affine space of dimension
 $d(\K)m(d-nk)$. 
 Furthermore, it is easy to check that the space
$\SZ_k\setminus\SZ_{k-1}$ is a locally trivial real affine bundle over $C_{k}(\R).$ 
Thus we see that it is a locally trivial real affine bundle over $C_{k}(\R)$ over rank 
$l_{d,k}=d(\K)m(d-nk)+k-1$.
\end{proof}

\begin{lemma}\label{lemma: E1}
There is a natural isomorphism
$$
E^{1;d}_{k,s}\cong
\begin{cases}
\Z & \mbox{if }s=(d(\K)mn-1)k
\mbox{ and } 0\leq k\leq  \lfloor \frac{d}{n}\rfloor,
\\
0 & \mbox{otherwise.}
\end{cases}
$$
\end{lemma}
\begin{proof}
If $k\leq 0$, the assertion is trivial.
If $k>\lfloor \frac{d}{n}\rfloor$, the assertion easily follows from
(\ref{filt}).
So
suppose that $1\leq k\leq \lfloor \frac{d}{n}\rfloor$.
Since there is a homeomorphism
$C_k(\R)\cong \R^k$ and it is contractible, the
affine bundle $\xi_{d,k}$ is trivial.
Hence, there is a homeomorphism
\begin{equation}\label{eq: ldk}
(\SZ_k)_+/(\SZ_{k-1})_+
\cong
(\SZ_k\setminus\SZ_{k-1})_+
\cong (\R^{l_{d,k}}\times \R^k)_+=
S^{k+l_{d,k}}.
\end{equation}
Thus there is an isomorphism
\begin{align*}
E^{1;d}_{k,s}
&\cong 
\tilde{H}^{d(\K)md+k-s-1}(S^{k+l_{d,k}};\Z)
\cong
\begin{cases}
\Z & \mbox{ if }s=(d(\K)mn-1)k
\\
0 & \mbox{ otherwise}
\end{cases}
\end{align*}
and this completes the proof.
\end{proof}

\begin{corollary}\label{crl: homology}
If $d(\K)mn\geq 3$,
there is an isomorphism
$$
H_k(\Q^{d,m}_{n}(\K);\Z)\cong
\begin{cases}
\Z & \mbox{ if }k=(d(\K)mn-2)i\ \mbox{ and }\  0\leq i\leq \lfloor \frac{d}{n}\rfloor ,
\\
0 & \mbox{ otherwise.}
\end{cases}
$$
\end{corollary}
\begin{proof}
Consider the spectral sequence
$$
\{E^{t;d}_{k,s},d^t:E^{t;d}_{k,s}\to E^{t;d}_{k+t,s+t-1}\}
\Rightarrow H_{s-k}(\Q^{d,m}_{n}(\K);\Z).
$$
Then, for dimensional reasons and by
Lemma \ref{lemma: E1}, it is easy to see that $E^{1;d}_{**}=E^{\infty;d}_{**}$ and the result follows.
\end{proof}

\begin{lemma}\label{lmm: 1-connected}
If $d(\K)mn\geq 4$, the space $\Q^{d,m}_{n}(\K)$ is simply connected.
\end{lemma}
\begin{proof}
If $n=1$ or $m=1$, the assertion follows from (i) and (ii) of Theorem \ref{thm: KY1-I}, and
suppose that  $m\geq 2$ and $n\geq 2$.
An element of $\pi_1(\Q^{d,m}_{n}(\K))$ can be represented by
strings with total multiplicity $d$ of $m$ different colors similarly to the  classical representation of the elements of the braid group 
$\mbox{Br}_d=\pi_1(C_d(\C))$
\cite{Hansen}.
When all strings of $m$ different colors moves continuously,
the following case  is not allowed to occur in this representation:
\begin{enumerate}
\item[$(\dagger)$]
All strings of multiplicity $\geq n$ of $m$ different colors pass through a single point on the real line.
\end{enumerate}
By using this representation,
one can show that any strings can intersect, pass through one another 
and thus can change the order
(see also \cite[\S 5. Appendix]{GKY1} for more details). 
Thus we see  that
$a\cdot  b=b\cdot a$ for any $a,b\in \pi_1(\Q^{d,m}_{n}(\K))$
and we know that
$\pi_1(\Q^{d,m}_{n}(\K))$ is an abelian group.
Moreover, since
$d(\K)mn\geq 4$, by Corollary \ref{crl: homology}
we see that
$H_1(\Q^{d,m}_n(\K);\Z)=0$. 
Hence, we have an isomorphism
$
\pi_1(\Q^{d,m}_{n}(\K))\cong H_1(\Q^{d,m}_{n}(\K);\Z)=0$, and the assertion follows.
\end{proof}

\section{Stabilization maps}\label{section: sd}

\begin{definition}\label{def: stabilization}
We identify $\C=\R^2$ by the identification
$x+y\sqrt{-1}\mapsto (x,y)$.
\par
(i)
For each integer $d\geq 1$ let $\varphi_d:\C\to \mathbb{H}_d=
\{x\in \C :\mbox{Re }(x)<d\}$ be any fixed homeomorphism
such that $\varphi (\overline{\alpha})=\overline{\varphi(\alpha)}$
for any $\alpha\in \C$.
Note that
$\varphi (x)\in \R$ and $\varphi (x)<d$ if $x\in \R$.
Let
$\tilde{\varphi}_d:\P^d(\C)\to \P^{d}(\C)$ denote the map given
by
\begin{equation}
\tilde{\varphi}_d(\prod_{j=1}^d(z-\alpha_j))=\prod_{j=1}^d(z-\varphi_d(\alpha_j)).
\end{equation}
\par
(ii)
For each integer $d\geq 1$, let $D_{d,\K}\subset \K^m$ denote the open set
defined by
\begin{equation}\label{eq: region}
D_{d,\K}=
\begin{cases}
\{(\alpha_1,\cdots ,\alpha_m)\in \R^m:d<\alpha_k<d+1,\alpha_i\not= \alpha_j\mbox{ if }i\not= j\} & \mbox{ if }\K=\R
\\
\{(\alpha_1,\cdots ,\alpha_m)\in \C^m:d<\mbox{Re}(\alpha_k)<d+1,
\alpha_i\not= \alpha_j\mbox{ if }i\not= j\} & \mbox{ if }\K=\C
\end{cases}
\end{equation}
and
fix some point $\textit{\textbf{x}}_d=(x_{d,1},\cdots ,x_{d,m})\in D_{d,\K}$.
Then as in \cite[page 42]{Se}
we can define {\it the stabilization map} 
$s^{d,m}_{n,\K}:\Q^{d,m}_{n}(\K)\to \Q^{d+1,m}_{n}(\K)$
by
\begin{equation}\label{eq: stab-infty}
s^{d,m}_{n,\K}(f)=
((z-x_{d,1})\tilde{\varphi}_d(f_1(z)),\cdots
,((z-x_{d,m})\tilde{\varphi}_d(f_m(z)))
\end{equation}
for $f=(f_1(z),\cdots ,f_m(z))\in \Q^{d,m}_n(\K)$. 
Note that the homotopy class of
the map $s^{d,m}_{n,\K}$ does not depend on the choice of the homeomorphism $\varphi_d$
and the 
point $\textit{\textbf{x}}_d.$
\par
(iii) Let $\Q^{\infty, m}_{n}(\K)$ denote the colimit
$\dis \lim_{d\to\infty}\Q^{d,m}_{n}(\K)$
constructed from the stabilization maps $s^{d,m}_{n,\K}$,
$$
\Q^{1,m}_{n}(\K)\stackrel{s^{1,m}_{n,\K}}{\longrightarrow}
\Q^{2,m}_{n}(\K)\stackrel{s^{2,m}_{n,\K}}{\longrightarrow}
\Q^{3,m}_{n}(\K)
\stackrel{s^{3,m}_{n,\K}}{\longrightarrow}
\Q^{4,m}_{n}(\K)\stackrel{s^{4,m}_{n,\K}}{\longrightarrow}
\Q^{5,m}_{n}(\K)\stackrel{s^{5,m}_{n,\K}}{\longrightarrow}
\cdots
\cdots
$$
\par
(iv)
Let $f^{d,m}_{n,\K}:\Q^{d,m}_{n}(\K)\to
\Q^{d+2,m}_n(\K)$ denote the map defined by
\begin{equation}
f^{d,m}_{n,\K}(f)
=((z^2+1)f_1(z),(z^2+1)f_2(z),\cdots ,(z^2+1)f_m(z))
\end{equation}
for $f=(f_1(z),\cdots ,f_m(z))\in \Q^{d,m}_{n}(\K).$
\qed
\end{definition}
\begin{lemma}\label{lemma: stabilization map f}
$f^{d,m}_{n,\K}\simeq s^{d+1,m}_{n,\K}\circ s^{d,m}_{n,\K}$
up to homotopy equivalence.
\end{lemma}
\begin{proof}
Let $f=(f_1(z),\cdots ,f_m(z))\in \Q^{d,m}_n(\K)$.
Then 
the map $s^{d+1,m}_{n,k}\circ s^{d,m}_{n,\K}(f)$
is represented up to homotopy by
$
\big((z-x_{d,1})(z-x_{d+1,1})\overline{f}_1,
\cdots ,((z-x_{d,m})(z-x_{d+1,m})\overline{f}_m\big),
$
where 
$\overline{f}_k=\tilde{\varphi}_d(f_k(z))$ for $1\leq k\leq m$.
\par
Define the family of monic polynomials $\{\phi_t(z):0\leq t\leq 1\}$ 
of degree $2$ in $\K [z]$  by
{\small
$$
\phi_t(z)=
\begin{cases}
\big(z-x_{d,1}\big)\big(z-2tx_{d,1}-(1-2t)x_{d+1,1}\big)
& \mbox{if }0\leq t\leq \frac{1}{2}
\\
\big(z-(2-2t)x_{d,1}-(2t-1)\sqrt{-1}\big)\big(z-(2-2t)x_{d,1}+(2t-1)\sqrt{-1}\big)
& \mbox{if }\frac{1}{2}\leq t\leq 1
\end{cases}
$$
}
\newline
and consider the homotopy
$F:\Q^{d,m}_n(\K)\times [0,1]\to \Q^{d+2,m}_n(\K)$ given by
$$
F(f,t)=
\big(\phi_t(z)\overline{f}_1,
(z-x_{d,2})(z-x_{d+1,2})\overline{f}_2,\cdots ,
(z-x_{d,m})(z-x_{d+1,m})\overline{f}_m\big).
$$
Since $F(f,0)=s^{d+1,m}_{n,\K}\circ s^{d,m}_{n,\K}(f)$,
the map
$s^{d+1,m}_{n,\K}\circ s^{d,m}_{n,\K}(f)$ is homotopic to the map 
$$
F(f,1)=
\big((z^2+1)\overline{f}_1,
(z-x_{d,2})(z-x_{d+1,2})\overline{f}_2,\cdots ,
(z-x_{d,m})(z-x_{d+1,m})\overline{f}_m\big).
$$ 
If we use this trick again, the map
$s^{d+1,m}_{n,\K}\circ s^{d,m}_{n,\K}(f)$ is homotopic to the map given by
$$
\big((z^2+1)\overline{f}_1,(z^2+1)\overline{f}_2,
(z-x_{d,3})(z-x_{d+1,3})\overline{f}_3,\cdots ,
(z-x_{d,m})(z-x_{d+1,m})\overline{f}_m\big).
$$
By using this trick repeatedly, we see that
the map
$s^{d+1,m}_{n,k}\circ s^{d,m}_{n,\K}(f)$ is homotopic to the map
$
\big((z^2+1)\overline{f}_1,(z^2+1)\overline{f}_2,
(z^2+1)\overline{f}_3,\cdots ,(z^2+1)\overline{f}_m\big)$,
which is also homotopic to the map
$f^{d,m}_{n,\K}(f)$.
This completes the proof.
\end{proof}

\begin{corollary}\label{crl: colimit}
Let $\epsilon\in \{0,1\}$.
Then the space $\Q^{\infty,m}_n(\K)$ is homotopy equivalent to the colimit
constructed from the stabilization maps $f^{d,m}_{n,\K}$,
$$
\Q^{1+\epsilon,m}_{n}(\K)
\stackrel{f^{1+\epsilon,m}_{n,\K}}{\longrightarrow}
\Q^{3+\epsilon,m}_{n}(\K)\stackrel{f^{3+\epsilon,m}_{n,\K}}{\longrightarrow}
\Q^{5+\epsilon,m}_{n}(\K)
\stackrel{f^{5+\epsilon,m}_{n,\K}}{\longrightarrow}
\Q^{7+\epsilon,m}_{n}(\K)\stackrel{f^{7+\epsilon,m}_{n,\K}}{\longrightarrow}
\cdots
\cdots
$$
\end{corollary}
\begin{proof}
This follows from the definition of $\Q^{\infty,m}_n(\K)$ and Lemma
\ref{lemma: stabilization map f}.
\end{proof}
Let $\textit{\textbf{x}}_d=(x_{d,1},\cdots ,x_{d,m})\in D_{d,\K}$ and
consider the stabilization map $s^{d,m}_{n,\K}$ given by
(\ref{eq: stab-infty}).
Note that
this map  clearly extends to a map
$\overline{s}^{d,m}_{n,\K}:\P^d(\K)^m\to\P^{d+1}(\K)^m$ 
by the same formula 
and
its restriction gives a stabilization map
$\tilde{s}^{d,m}_{n,\K}:\Sigma^{d,m}_{n,\K}\to \Sigma^{d+1,m}_{n,\K}$
between discriminants.
If we choose a sufficiently small positive number $\epsilon_0 >0$,
then the following condition is satisfied:
\begin{equation*}
O(x_{d,i})\cap O(x_{d,j})=\emptyset
\mbox{ if }i\not=j,
\mbox{ where }O(x)=\{\alpha \in \K:
\vert x-\alpha \vert <\epsilon_0\}
\mbox{ for }x\in \K.
\end{equation*}
For each element
$\textit{\textbf{y}}=(y_1,\cdots ,y_m)\in 
O(x_{d,1})\times \cdots \times
O(x_{d,m})$, if we consider the map
$$
\begin{CD}
\Sigma^{d,m}_{n,\K} @>>>\Sigma^{d+1,m}_{n,\K}
\\
(f_1(z),\cdots ,f_m(z)) @>>>
((z-y_1)\tilde{\varphi}_d(f_1(z)),\cdots ,(z-y_m)\tilde{\varphi}_d(f_m(z)))
\end{CD}
$$
then it is easy to see that this map coincides with
the map $\tilde{s}^{d,m}_{n,\K}$
if $\textit{\textbf{y}}=\textit{\textbf{x}}_d$.
Thus, by using a homeomorphism $O(x_{d,k})\cong \K$
for $1\leq k\leq m$, we see that
the map
$\tilde{s}^{d,m}_{n,\K}:\Sigma^{d,m}_{n,\K}\to \Sigma^{d+1,m}_{n,\K}$
 extends to an open embedding
\begin{equation}\label{equ: open-stab}
\tilde{s}^{d,m}_{n,\K}:\Sigma^{d,m}_{n,\K}\times\K^m\to 
\Sigma^{d+1,m}_{n,\K}.
\end{equation}
Since one-point compactification is contravariant for open embeddings,
it
induces a map
\begin{equation}\label{equ: embedding3}
\tilde{s}^{d,m}_{n,\K+}: 
(\Sigma^{d+1,m}_{n,\K})_+
\to
(\Sigma^{d,m}_{n,\K}\times \K^{m})_+=
(\Sigma^{d,m}_{n,\K})_+\wedge S^{d(\K)m}
\end{equation}
between one-point compactifications 
and
we obtain  the following  commutative diagram
\begin{equation}\label{diagram: discriminant}
\begin{CD}
\tilde{H}_k(\Q^{d,m}_{n}(\K);\Z) 
@>{s^{d,m}_{n,\K*}}>>\tilde{H}_k(\Q^{d+1,m}_{n}(\K);\Z)
\\
@V{Al}V{\cong}V @V{Al}V{\cong}V
\\
H^{d(\K)dm-k-1}_c(\Sigma^{d,m}_{n,\K};\Z) 
@>\tilde{s}^{d,m*}_{n+}>>
H^{d(\K)(d+1)m-k-1}_c(\Sigma^{d+1,m}_{n,\K};\Z)
\end{CD}
\end{equation}
where $Al$ denotes the Alexander duality isomorphism and
 $\tilde{s}^{d,m*}_{n+}$ the composite of the
the suspension isomorphism with the homomorphism
${(s^{d,m}_{n,\K})^*}$,
{\small
$$
H^{d(\K)dm-k-1}_c(\Sigma^{d,m}_{n,\K};\Z)
\stackrel{\cong}{\rightarrow}
H^{d(\K)(d+1)m-k-1}_c(\Sigma^{d,m}_{n,\K}\times \K^{m};\Z)
\stackrel{(\tilde{s}^{d,m}_{n,\K})^*}{\longrightarrow}
H^{d(\K)(d+1)m-k-1}_c(\Sigma^{d+1,m}_{n,\K};\Z).
$$
}
Note that the map $\tilde{s}^{d,m}_{n,\K}$  induces the filtration preserving
map 
\begin{equation}
\hat{s}^{d,m}_{n,\K}:\SZ (\K)\times \K^m\to \mathcal{X}^{d+1}(\K)
\end{equation}
and
it induces the homomorphism of spectral sequences
\begin{equation}\label{equ: theta1}
\{ \theta_{k,s}^t:E^{t;d}_{k,s}\to E^{t;d+1}_{k,s}\}.
\end{equation}
\begin{lemma}\label{lmm: E1}
If $d(\K)mn\geq 4$ and
$1\leq k\leq \lfloor \frac{d}{n}\rfloor$, 
$\theta^{\infty}_{k,s}:
E^{\infty;d}_{k,s}\stackrel{\cong}{\rightarrow} 
{E}^{\infty;d+1}_{k,s}$ is
an isomorphism for any $s$.
\end{lemma}
\begin{proof}
Since we use the result of Lemma \ref{lemma: vector bundle}
and $\SZ_k=\SZ_{k-1}$ if $k>\lfloor\frac{d}{n}\rfloor$,
suppose that $1\leq k\leq \lfloor \frac{d}{n}\rfloor$.
If we set 
$\hat{s}^{d,m}_{n;k}=\hat{s}^{d,m}_{n,\K}\vert
\SZ_k\setminus\SZ_{k-1},$ it follows from the construction of
the map $\hat{s}^{d,m}_{n,\K}$ that
the following diagram is commutative:
\begin{equation*}
\begin{CD}
(\SZ_k\setminus\SZ_{k-1})\times \K^{m} @>\pi_{k,d}>> C_{k}(\R)
\\
@V{\hat{s}^{d,m}_{n;k}}VV \Vert @.
\\
\mathcal{X}^{d+1}_k\setminus 
\mathcal{X}^{d+1}_{k-1} 
@>\pi_{k,d+1}>> C_{k}(\R)
\end{CD}
\end{equation*}
Since one-point compactification is contravariant for open embeddings,
the map $\hat{s}^{d,m}_{n;k}$ induces the map
$
(\hat{s}^{d,m}_{n;k})_+:
(\mathcal{X}^{d+1}_k\setminus 
\mathcal{X}^{d+1}_{k-1})_+\to 
((\SZ_k\setminus\SZ_{k-1})\times \K^{m})_+=
(\SZ_k\setminus\SZ_{k-1})_+\wedge
S^{d(\K)m} 
$
between one-point compactfications.
Recall from the proof of Lemma \ref{lemma: vector bundle} that
 $\xi_{d,k}$ (resp. $\xi_{d+1,k}$) is a trivial real affine bundle over $C_k(\R)$
 with rank $l_{d,k}$ (resp. $l_{d+1,k}$).
 Moreover, note that
\begin{align*}
(d(\K)m(d+1)+k-s-1)-l_{d,k}-d(\K)m&=
(d(\K)m(d+1)+k-s-1)-l_{d+1,k}
\\
&=d(\K)mnk-s.
\end{align*}
Then from the above commutative diagram and the suspension isomorphism,
we obtain the following commutative diagram:
{\small
$$
\begin{CD}
E^{1;d}_{k,s} @>\theta^1_{k,s}>> E^{1;d+1}_{k,s}
\\
\Vert @. \Vert @.
\\
H^{d(\K)md+k-s-1}_c
(\SZ_k\setminus\SZ_{k-1};\Z)
 @.
H^{d(\K)m(d+1)+k-s-1}_c(
\mathcal{X}^{d+1}_k\setminus 
\mathcal{X}^{d+1}_{k-1};\Z)
\\
@V{sus}V{\cong}V \Vert @.
\\
H^{d(\K)m(d+1)+k-s-1}_c
((\mathcal{X}^{d}_k\setminus 
\mathcal{X}^{d}_{k-1})\times \K^m;\Z) 
@>(\hat{s}^{d,m}_{n;k})_+^*>>
H^{d(\K)m(d+1)+k-s-1}_c(
\mathcal{X}^{d+1}_k\setminus 
\mathcal{X}^{d+1}_{k-1};\Z)
\\
@V{}V{\cong}V @V{}V{\cong}V
\\
H^{d(\K)mnk-s}_c(C_k(\R);\Z) @>>=> 
H^{d(\K)mnk-s}_c(C_k(\R);\Z)
\end{CD}
$$
}
\par
Hence, $\theta^1_{k,s}$ is an isomorphism for any $s$.
Since the degree of the differential $d^t$ is $(t,t-1)$, 
it follows from Lemma \ref{lemma: E1} and dimensional reasons
that the spectral sequence collapses at $E^1$-term.
Hence,
$E^{1;d+\epsilon}_{k,*}=E^{\infty;d+\epsilon}_{k,*}$ 
for $\epsilon\in\{0,1\}$ and the assertion follows.
This completes the proof of Lemma \ref{lmm: E1}.
\end{proof}

Now we prove the key result.
\begin{theorem}\label{thm: stab1}
If $d(\K)mn\geq 4$,
the stabilization map 
$$
s^{d,m}_{n,\K}:\Q^{d,m}_{n}(\K)\to
\Q^{d+1,m}_{n}(\K)
$$ 
is a homology equivalence for 
$\lfloor \frac{d}{n}\rfloor =\lfloor\frac{d+1}{n}\rfloor$, and it is a
homology equivalence through dimension
$D(d;m,n,\K)$ for
$\lfloor \frac{d}{n}\rfloor <\lfloor\frac{d+1}{n}\rfloor$.
\end{theorem}
\begin{proof}
First, consider the case 
$\lfloor \frac{d}{n}\rfloor =\lfloor\frac{d+1}{n}\rfloor$.
In this case, by using Lemma \ref{lemma: E1} and 
Lemma \ref{lmm: E1} we easily see that
$\theta^{\infty}_{k,s}:E^{\infty;d}_{k,s}\stackrel{\cong}{\longrightarrow}
E^{\infty;d+1}_{k,s}$ 
is an isomorphism for any $(k,s)$.
Since $\theta^t_{k,s}$ is induced from $\hat{s}^{d,m}_{n,\K}$,
it follows from (\ref{diagram: discriminant}) that
the map $s^{d,m}_{n,\K}$ is a homology equivalence.
\par
Next assume that
$\lfloor \frac{d}{n}\rfloor <\lfloor\frac{d+1}{n}\rfloor$,
i.e. $\lfloor \frac{d+1}{n}\rfloor =\lfloor\frac{d}{n}\rfloor +1.$
In this case,
by considering the differential
$d^t:E^{t;d+\epsilon}_{k,s}\to E^{t;d+\epsilon}_{k+t,s+t-1}$
($\epsilon \in \{0,1\})$, Lemma \ref{lmm: E1}
and Lemma \ref{lemma: E1}, we easily see that
$\theta^{\infty}_{k,s}:E^{\infty;d}_{k,s}\to E^{\infty;d+1}_{k,s}$ 
is an isomorphism for any $(k,s)$ as long as
the condition $s-k\leq D(d;m,n,\K)$ is satisfied.
Hence, 
the map $s^{d,m}_{n,\K}$ is a homology equivalence
through dimension $D(d;m,n,\K)$.
\end{proof}

\begin{definition}
(i)
For an $m$-tuple $D=(d_1,\cdots ,d_m)\in \N^m$ of positive integers,
let
$\Q^{D;m}_n(\K)=\Q^{d_1,\cdots ,d_m;m}_n(\K)$ denote the space of all
$m$-tuples $(f_1(z),\cdots ,f_m(z))\in \P^{d_1}(\K)\times
\cdots \times \P^{d_m}(\K)$
such that the
polynomials $\{f_k(z)\}_{k=1}^m$ have no common {\it real}
root of multiplicity $\geq n$
(but they may have a common {\it complex} root of any multiplicity).
Note that
\begin{equation}
\Q^{D;m}_n(\K)
=
\Q^{d_1,\cdots ,d_m;m}_n(\K)
=
\Q^{d,m}_n(\K)
\qquad
\mbox{ if }
d=d_1=d_2=\cdots =d_m.
\end{equation}
\par
(ii)
Let us choose any fixed real number $x_0\in (d,d+1)$.
We
define the $i$-th stabilization map
$s^{D;i}_{n,\K}:
\Q^{d_1,\cdots ,d_m;m}_n(\K)\to \Q^{d_1,\cdots ,d_{i-1},d_i+1,d_{i+1},\cdots ,d_m;m}_n(\K)$
by
\begin{equation}\label{eq: stab-infty2}
s^{D;i}_{n,\K}(f)=
(\overline{f}_1(z),\cdots,\overline{f}_{i-1}(z),
(z-x_0)\overline{f}_{i}(z),
\overline{f}_{i+1}(z),
\cdots
,\overline{f}_{m}(z))
\end{equation}
for $f=(f_1(z),\cdots ,f_m(z))\in \Q^{D;m}_n(\K)$,
where
$\overline{f}_k(z)=\tilde{\varphi}_d(f_k(z))$
for $1\leq k\leq m$.
\qed
\end{definition}
\begin{theorem}\label{thm: stab1ity2}
If $d(\K)mn\geq 4$,
the stabilization map 
$$
s^{D;i}_{n,\K}:
\Q^{d_1,\cdots ,d_m;m}_{n}(\K)\to
\Q^{d_1,\cdots ,d_{i-1},d_i+1,d_{i+1},\cdots ,d_m;m}_{n}(\K)
$$ 
is a homology equivalence for 
$\lfloor \frac{d_i}{n}\rfloor =\lfloor\frac{d_i+1}{n}\rfloor$, and it is a
homology equivalence through dimension
$D(d_i;m,n,\K)$ for
$\lfloor \frac{d_i}{n}\rfloor <\lfloor\frac{d_i+1}{n}\rfloor$.
\end{theorem}
\begin{proof}
The proof is completely analogous to that of Theorem \ref{thm: stab1}.
So we omit the details. 
\end{proof}
\begin{lemma}\label{lmm: abelian}
If $d(\K)mn\geq 4$, 
$\pi_1(\Q^{d_1,\cdots ,d_m;m}_n(\K))$
is an abelian group.
\end{lemma}
\begin{proof}
We can prove the assertion in the same way as in
Lemma \ref{lmm: 1-connected}
by using the string representation. 
\end{proof}

\section{Configuration spaces and horizontal scanning maps}
\label{section: scanning maps}

In this section, we prove the stable result (see Theorem \ref{thm: natural map} below)
by using  \lq\lq horizontal scanning maps\rq\rq.
We continue to assume that $\K =\R$ or $\C$.

\begin{definition}
(i)
For a space $X$ let $\SP^d(X)$ denote the $d$-th {\it symmetric product}
defined by the quotient space 
\begin{equation}
\SP^d(X)=X^d/S_d,
\end{equation}
where
the symmetric group $S_d$ of $d$-letters acts on $X^d$ by the
permutation of coordinates.
Since $F(X,d)$ is a $S_d$-invariant subspace of $X^d$ and
$C_d(X)=F(X,d)/S_d$, there is a natural inclusion
$C_d(X)\subset \SP^d(X).$
\par
(ii)
It is easy to see that
an element $\alpha\in\SP^d(X)$ may be identified with
the formal linear combination
\begin{equation}\label{combination}
\alpha =\sum_{i=1}^kd_ix_i,
\quad\quad
\mbox{where }\{x_i\}_{i=1}^k\in C_k(X)\ \  \mbox{and }\  \sum_{i=1}^kd_i=d.
\end{equation}
We shall refer to $\alpha$
as {\it a configuration} (or {\it divisor}) 
having 
{\it a multiplicity}
$d_i$ at the point $x_i$.
\par
(iii)
Note that there is a natural homeomorphism
\begin{equation}\label{eq: varphi}
\Phi_d:\P^d(\C)\stackrel{\cong}{\longrightarrow}\SP^d(\C)
\quad
\mbox{given by}\quad
\Phi_d(\prod_{i=1}^k(z-\alpha_i)^{d_i})= \sum_{i=1}^{k}d_i\alpha_i.
\end{equation}
Since there is a natural inclusion
$\P^d(\R)\subset \P^d(\C)$,
we define a subspace
$\SP^d_{\R}\subset \SP^d(\C)$ as the image
\begin{equation}
\SP^d_{\R}=\Phi_d(\P^d(\R)).
\end{equation}
Note that restriction gives a homeomorphism
\begin{equation}
\Phi_d\vert \P^d(\R):\P^d(\R)
\stackrel{\cong}{\longrightarrow}
\SP^d_{\R}.
\end{equation}
\par
(iv)
For a subspace $X\subset \C$, let
$\SP^d_{\R}(X)\subset \SP^d(X)$ denote the subspace
of $\SP^d(X)$
given by
\begin{equation}
\SP^d_{\R}(X)=\SP^d_{\R}\cap\SP^d(X).
\end{equation}
\par
(v)
Similarly, for an $m$-tuple $D=(d_1,\cdots ,d_m)\in \N^m$
of positive integers,
define subspaces 
$Q^{D;m}_{n,\R}(X)\subset Q^{D;m}_{n,\C}(X)$ in
$\SP^{d_1}(X)\times \cdots \times \SP^{d_m}(X)$ by
\begin{equation}\label{equ: Q(X)}
\begin{cases}
Q^{D;m}_{n,\C}(X)&=
Q^{d_1,\cdots ,d_m;m}_{n,\C}(X)
\\
&=\big\{(\xi_1,\cdots ,\xi_m)\in\SP^{d_1}(X)\times \cdots
\times \SP^{d_m}(X):
(*)_n\big\},
\\
Q^{D;m}_{n,\R}(X)
&=
Q^{d_1,\cdots ,d_m;m}_{n,\R}(X)
=Q^{D;m}_{n,\C}(X)\cap (\SP^{d_1}_{\R}(X)\times \cdots \times
(\SP^{d_m}_{\R}(X)),
\end{cases} 
\end{equation}
where the condition $(*)_n$ is given by
\begin{enumerate}
\item[$(*)_n$]
\qquad
The element $(\cap_{k=1}^m\xi_k)\cap \R$ contains no point of
multiplicity $\geq n$.
\end{enumerate}
When $d=d_1=\cdots =d_m$ and $D=(d_1,\cdots ,d_m)$,
we set
$Q^{d,m}_{n,\K}(X)=Q^{D;m}_{n,\K}(X)$.
\qed
\end{definition}
\begin{remark}\label{rmk: Q}
(i)
For an open set $V\subset \C$
and an $m$-tuple $D=(d_1,\cdots ,d_m)\in \N^m$,
 let $\Q^{D;m}_{n}(V;\K)$ denote the subspace
of $\Q^{D;m}_n(\K)$ consisting of all $m$-tuples
$(f_1(z),\cdots ,f_m(z))\in \Q^{D;m}_n(\K)$ such that all roots of $f_k(z)$ lie
in $V$ for each $1\leq k\leq m$.
Then it is easy to see that there is a homeomorphism 
$\Q^{D;m}_{n}(V;\K)\cong Q^{D;m}_{n,\K}(V)=Q^{d_1,\cdots ,d_m;m}_{n,\K}(V)$ given by
\begin{align}\label{eq: QK}
\big(\prod_{k=1}^{s_1}(z-\alpha_{k}^{(1)})^{d_{k}^{(1)}},\cdots 
,\prod_{k=1}^{s_1}(z-\alpha_{k}^{(m)})^{d_{k}^{(m)}}\big)
\mapsto
\big(\sum_{k=1}^{s_1}d_{k}^{(1)}\alpha_{k}^{(1)},\cdots ,
\sum_{k=1}^{s_m}d_{k}^{(m)}\alpha_{k}^{(m)}\big),
\end{align}
where $\alpha_k^{(i)}\in V$ 
and $\sum_{k=1}^{s_i}d_k^{(i)}=d_i$
for $1\leq k\leq s_i$ and $1\leq i\leq m$.
\par
(ii)
Let $X\subset \C$ be a subspace.
If we use the notation (\ref{combination}), an element
$\alpha \in \SP^d_{\R}(X)$
can be represented as the formal linear combination
\begin{equation}
\alpha =\sum_{i=1}^sd_ix_i +
\sum_{j=1}^kd_j^{\prime}(\alpha_j+\overline{\alpha}_j),
\end{equation}
where 
$\{x_i\}_{i=1}^s\in C_s(\R)$, 
$\{\alpha_j\}_{j=1}^k\in C_k(\mathbb{H} \cap (X\setminus \R))$,
$\sum_{i=1}^sd_i+2\sum_{j=1}^kd_j^{\prime}=d$
and
$\mathbb{H}=\{\alpha \in \C:\mbox{Im }(\alpha)>0\}.$
\newline
If $X\cap \R=\emptyset$ and
$D=(d_1,\cdots ,d_m)\in \N^m$,
$Q^{D;m}_{n,\K}(X)=\prod_{k=1}^m\SP^{d_k}_{\K}(X)$, where
we set $\SP^d_{\K}=\SP^d(X)$ if $\K =\C$.
\qed
\end{remark}

\begin{definition}
Let $X\subset \C$ and  
$\emptyset\not= A\subset X$ be a closed subspace.
\par
(i) Let 
$\SP^d(X,A)$ denote the quotient space
$\SP^d(X)/\sim ,$ 
where
the equivalence relation \lq\lq$\sim$\rq\rq \ on $\SP^d(X)$
is given by
\begin{equation}
\xi \sim \eta \quad \mbox{ if }\quad
\xi\cap (X\setminus A)=\eta \cap (X\setminus A)
\quad
\mbox{ for }\xi,\eta\in\SP^d(X).
\end{equation}
Thus, the points in $A$ are ignored.
Note that
there is a natural inclusion
$\SP^d(X,A)\subset \SP^{d+1}(X,A)$ by
adding a point in $A$.
Let
$\SP(X,A)$ denote the space given by the union
\begin{equation}
\SP (X,A)=\bigcup_{d\geq 0}\SP^d(X,A),
\ \mbox{ where we set }\ \SP^0(X,A)=\{\emptyset\}
\ \mbox{ for }d=0.
\end{equation}
\par
(ii)
Similarly,
let $\Q^{d,m}_{n,\K}(X,A)$ be the quotient space of the space
\begin{equation}
\big\{(\xi_1,\cdots ,\xi_m)\in\SP^{d_1}(X)\times \cdots
\times \SP^{d_m}(X):
(*)_{n,A}\big\}
\end{equation}

where the condition $(*)_{n,A}$ is given by
\begin{enumerate}
\item[$(*)_{n,A}$]
\qquad
The element $(\cap_{k=1}^m\xi_k)\cap \R \cap (X\setminus A) $ contains no point of
multiplicity $\geq n$.
\end{enumerate}
and the equivalence relation is given by 
$$
(\xi_1,\cdots ,\xi_m)\sim
(\eta_1,\cdots ,\eta_m)
\quad \mbox{ if}\quad \xi_i\cap (X\setminus A)=\eta_i \cap (X\setminus A)
\ \mbox{ for each }1\leq i\leq m.
$$
Thus, points in $A$ are ignored.
Note that
there is a natural inclusion
\begin{equation}\label{eq: subset}
Q^{d,m}_{n,\K}(X,A)\subset Q^{d+1,m}_{n,\K}(X,A)
\end{equation}
given by adding points in $A$.
Define the space $Q^{m}_{n,\K}(X,A)$ by the union
\begin{equation}\label{eq: union}
Q^{m}_{n,\K}(X,A)=\bigcup_{d\geq 0}Q^{d,m}_{n,\K}(X,A),
\ \ 
\mbox{where we set }\ \ 
Q^{0,m}_{n,\K}(X,A)=\{(\emptyset, \cdots ,\emptyset)\}.
\end{equation}
\end{definition}

\begin{remark}
As sets $Q^{m}_{n,\K}(X,A)$ and 
the disjoint union
$\coprod_{d\geq 0}Q^{d,m}_{n,\K}(X\setminus A)$ are 
bijectively equivalent, but they are not homeomorphic topological spaces.
For example, if $X$ is connected, then $Q^{m}_{n,\K}(X,A)$ is connected while the disjoint union is, in general, disconnected.
\qed
\end{remark}

We need two kinds of horizontal scanning maps.
First, we define the scanning map for the configuration space of
particles.
From now on, we make the identification $\C =\R^2$.

\begin{definition}
For a rectangle $X$ in $\C=\R^2$, let $\sigma X$ denote
the union of the sides of $X$ which are parallel to the $y$-axis, and
for a subspace $Z\subset \C=\R^2$, let $\overline{Z}$ be the closure of $Z$.
From now on, let $I$ denote the interval
$I=[-1,1]$
and
let $0<\epsilon <1$ be a fixed positive real number.
For each $x\in\R$, let $V(x)$ be the set defined by
\begin{equation}\label{eq: Vx}
V(x)=\{w\in\C: \vert \mbox{Re}(w)-x\vert  <\epsilon , \vert\mbox{Im}(w)\vert<1 \}
=(x-\epsilon ,x+\epsilon )\times (-1,1),
\end{equation}
and let us identify $I\times I=I^2$ with the closed unit rectangle
$\{t+s\sqrt{-1}\in \C: -1\leq t,s\leq 1\}$ 
in $\C$.
Now
define {\it the horizontal scanning map}
\begin{equation}
sc^{d,m}_n:\Q^{d,m}_{n}(\K )\to \Omega Q^{m}_{n,\K}(I^2,\partial I\times I)
=\Omega Q^{m}_{n,\K}(I^2,\sigma I^2)
\end{equation}
as follows.
For each $m$-tuple 
$\alpha =(\xi_1,\cdots ,\xi_m)\in Q^{d,m}_{n,\K}(\C)$
of configurations,
let
$
sc^{d,m}_n(\alpha):\R \to Q^{d,m}_{n,\K}(I^2,\partial I\times I)
=Q^{d,m}_{n,\K}(I^2,\sigma I^2)$
denote the map given by
\begin{align*}
\R\ni x
&\mapsto
(\xi_1\cap\overline{V}(x),\cdots ,\xi_m\cap\overline{V}(x))
\in
Q^{m}_{n,\K}(\overline{V}(x),\sigma \overline{V}(x))
\cong Q^{m}_{n,\K}(I^2,\sigma I^2),
\quad
\end{align*}
where 
we use the canonical identification
$(\overline{V}(x),\sigma \overline{V}(x))
\cong (I^2,\sigma I^2)=
(I^2,\partial I\times I).$
\par\vspace{2mm}\par
Since $\dis\lim_{x\to\pm \infty}sc^{d,m}_n(\alpha)(x)
=(\emptyset ,\cdots ,\emptyset)$, 
by setting $sc^{d,m}_n(\alpha)(\infty)=(\emptyset ,\cdots ,\emptyset)$
we obtain a based map
$sc^{d,m}_n(\alpha)\in \Omega Q^{m}_{n,\K}(I^2,\sigma I^2),$
where we identify $S^1=\R \cup \infty$ and
we choose the empty configuration
$(\emptyset ,\cdots ,\emptyset)$ as the base point of
$\Q^{m}_{n,\K}(I^2,\sigma I^2)$.
\par
Hence, if we identify $\Q^{d,m}_n(\K)=
Q^{d,m}_{n,\K}(\C)$,
 we obtain a map
$sc^{d,m}_n:
\Q^{d,m}_{n}(\K)\to 
\Omega Q^{m}_{n,\K}(I^2,\sigma I^2).
$
\par\vspace{1mm}\par
Since  $sc^{d+1,m}_{n,\K}\circ s^{d,m}_{n,\K}\simeq sc^{d,m}_{n,\K}$,
by setting $\dis S=\lim_{d\to\infty}sc^{d,m}_n$
we obtain {\it the stable horizontal scanning map}
\begin{equation}
S:
\Q^{\infty,m}_{n}(\K)
=\lim_{d\to\infty}\Q^{d,m}_{n}(\K)
\to 
\Omega
Q^{m}_{n,\K}(I^2,\partial I\times I)
=\Omega Q^m_{n,\K}(I^2,\sigma I^2).
\end{equation}
\end{definition}

\begin{theorem}[\cite{Gu}, \cite{Se}]\label{thm: scanning map}
If $d(\K)mn\geq 4$, the stable horizontal scanning map
$$
S:
\Q^{\infty,m}_{n}(\K)
\stackrel{\simeq}{\longrightarrow}
\Omega
Q^{m}_{n,\K}(I^2,\sigma I^2)
$$
is a homotopy equivalence.
\end{theorem}
\begin{remark}
Recall that
G. Segal proves that the scanning map
$S_G:\widehat{Q}
\stackrel{\simeq}{\longrightarrow}
\Omega^2(\CP^{\infty}\vee\CP^{\infty})$
is a homotopy equivalence in
\cite[\S 2]{Se}.
His proof consists of three steps.
First, he proves that there is a homotopy equivalence
$Q(S^2,\infty)\simeq \CP^{\infty}\vee \CP^{\infty}$ 
\cite[Prop. 3.1]{Se}, and secondly he shows that
{\it the horizontal scanning map}
$S_H:\widehat{Q}\stackrel{\simeq}{\longrightarrow}
\Omega Q(I^2,\sigma I^2)$ is a homotopy equivalence by using
\cite[Lemma 3.4]{Se}.
Finally, he proves that {\it the vertical scanning map} 
$S_V:Q(I^2,\sigma I^2)\stackrel{\simeq}{\longrightarrow}\Omega Q(S^2,\infty)$
is a homotopy equivalence in \cite[Prop. 3.2]{Se}.
Since Segal's scanning map $S_G$ is the composite
$\Omega S_V\circ S_H$ (up to homotopy equivalence), he concludes that it is a homotopy equivalence.
By using his method we will show that the map $S$ is a homotopy equivalence. 
\qed
\end{remark}
\begin{proof}
We identify $\C=\R^2$ by means of the identification
$x+\sqrt{-1}y\mapsto (x,y)$ in the usual way.
The proof is  analogous to the one given in
\cite[Prop. 3.2, Lemma 3.4]{Se} and \cite[Prop. 2]{Gu}.
Let $B$ and $B^*$ denote the rectangles in $\R^2=\C$ given by
$
B^*= [-1,2]\times [-1,1]$ and
$B = (0,1)\times (-1,1)$. 
Let $\{V_t:0<t<1\}$ be a family of open rectangles in $B$ given by
$V_t=(t-\epsilon (t),t+\epsilon (t))\times (-1,1)$, where
$\epsilon (t)$ denotes the continuous function defined on the open interval $(0,1)$
such that
$0<\epsilon (t)<\min \{t,1-t\}$
for any $t\in (0,1)$ with 
$\dis \lim_{t\to + 0}\epsilon (t)=\lim_{t\to 1-0}\epsilon (t)=0.$
\par
Let
$
\widetilde{sc}^{d,H}:
Q^{d,m}_{n,\K}(B)\times [0,1]\to Q^m_{n.\K}(\overline{B},\sigma \overline{B})
$
denote the map given by
$$
\widetilde{sc}^{d,H}((\xi_1,\cdots ,\xi_m),t)
=
(\xi_1\cap V_t,\cdots ,\xi_m\cap V_t)\in 
Q^{m}_{n,\K}(\overline{V_t},\sigma\overline{V_t})
\cong Q^m_{n.\K}(\overline{B},\sigma \overline{B}),
$$ 
where
we use the canonical identification
$
Q^{m}_{n,\K}(\overline{V_t},\sigma\overline{V_t})
\cong Q^m_{n,\K}(\overline{B},\sigma \overline{B}).$
Since $\dis \lim_{t\to + 0}\widetilde{sc}^{d,H}(\xi ,t)
=\lim_{t\to 1-0}\widetilde{sc}^{d,H}(\xi ,t)
=(\emptyset ,\cdots ,\emptyset)$ 
for any $\xi \in Q^{d,m}_{n,\K}(B)$,
the adjoint of $\widetilde{sc}^{d,H}$ defines the map
$sc^{d,H}:Q^{d,m}_{n,\K}(B)\to \Omega Q^m_{n,\K}(\overline{B},\sigma \overline{B})$.
If $s^{d}:Q^{d,m}_{n,\K}(B)\to Q^{d+1,m}_{n,\K}(B)$
denotes the stabilization map
defined by adding points from infinity as in Definition
\ref{def: stabilization},
by using the identification (\ref{eq: QK}) we obtain a homotopy commutative
diagram
$$
\begin{CD}
\Q^{d,m}_n(\K) @>s^{d,m}_n>> \Q^{d+1,m}_{n}(\K)
\\
@V{\cong}VV @V{\cong}VV
\\
Q^{d,m}_{n,\K}(B) @>s^{d}>> Q^{d+1,m}_{n,\K}(B)
\end{CD}
$$
Hence, if $Q^{\infty,m}_{n,\K}(B)$ denotes the colimit
$\dis Q^{\infty,m}_{n,\K}(B)=\lim_{d\to\infty}Q^{d,m}_{n,\K}(B)$
taken over the stabilization maps $s^{d}$, 
there is a commutative diagram
$$
\begin{CD}
\Q^{\infty ,m}_{n}(\K) @>S>>  \Omega Q^m_{n,\K}(I^2,\sigma I^2)
\\
@V{\cong}VV @V{\cong}VV
\\
Q^{\infty,m}_{n,\K}(B) @>S^{\p}>> \Omega Q^m_{n,\K}(\overline{B},\sigma \overline{B})
\end{CD}
$$
where we set $\dis S^{\p}=\lim_{d\to\infty}sc^{d,H}.$
So it suffices to prove that $S^{\p}$ is a homotopy equivalence.
\par\vspace{2mm}\par
To see this, let $J=(0,1)$ and
let us choose the sequence $\{\textbf{\textit{c}}_d\}_{d=1}^{\infty}$
of $m$-tuples 
$\textit{\textbf{c}}_d=
(c_{d,1},\cdots ,c_{d,m})
\in (J\times J)^m$ 
satisfying the following condition $(*)$:
\begin{enumerate}
\item[$(*)$]
$c_{d_1,i}\not= c_{d_2,j}$ if $(d_1,i)\not=(d_2,j)$,  
and
$\dis \lim_{d\to\infty}c_{d,i}=(1/2,1)\in \partial \overline{B}$
for each $1\leq i\leq m$.
\end{enumerate}
Let $\widehat{Q}^m_{n,\K}(B^*,\sigma B^*)$
denote the space of
$m$-tuples $(\xi_1,\cdots ,\xi_m)$
of
formal infinite divisors in $(B^*,\sigma B^*)$
satisfying the following two conditions:
\begin{enumerate}
\item[$(\dagger)_1$]
The element
$(\bigcap_{k=1}^m\xi_k)\cap \R \cap (B^*\setminus \sigma B^*)$
contains no point of multiplicity $\geq n$,
where
we identify $\R= (-\infty,\infty)\times \{0\}$.
\item[$(\dagger)_2$]
Each divisor $\xi_k$ is represented as the formal infinite sum of the form
$\xi_k=\sum_{d\geq 1}\xi_{k,d}$ 
such that $\xi_{k,d}\in \SP^1 (B^*,\sigma B^*)$ for each $d\geq 1$, and it
almost coincides (except finite sums) with
$\xi_{k}^*=\sum_{d\geq 1}(c_{d,k}+\overline{c_{d,k}})$.
When $\K =\R$,
the equality 
$\xi_k=\overline{\xi}_k$ also holds for each $1\leq k\leq m$, where we set
$\overline{\xi}_k=\sum_{d\geq 1}\overline{\xi_{k,d}}.$
\end{enumerate}
Similarly, let $\widehat{Q}^m_{n,\K}(B)$ denote
denote the space of
$m$-tuples $(\xi_1,\cdots ,\xi_m)$
of
formal infinite divisors in $B$ such that
the element
$(\bigcap_{k=1}^m\xi_k)\cap \R $
contains no point of multiplicity $\geq n$ and satisfies
the condition
$(\dagger)_2$. 
\par
Let us write $X_0=[-1,0]\times [-1,1]$ and $X_1=[1,2]\times [-1,1]$, and
note that $B^*=X_0\cup \overline{B}\cup X_1.$
Then we define
\begin{equation}
\hat{q}:\widehat{Q}^m_{n,\K}(B^*,\sigma B^*)
\to
Q^m_{n,\K}(\overline{B},\sigma \overline{B})^2
\end{equation}
to be the natural quotient map
$$
\widehat{Q}^m_{n,\K}(B^*,\sigma B^*)
\to
Q^m_{n,\K}(B^*,\sigma B^*\cup \overline{B})
\cong 
Q^m_{n,\K}(X_0,\sigma X_0)\times Q^m_{n,\K}(X_1,\sigma X_1)
\cong
Q^m_{n,\K}(\overline{B},\sigma \overline{B})^2.
$$
By using the Dold-Thom criterion 
\cite[Lemma 3.3]{McDuff}
(cf. \cite[Lemma 4K.3]{Hatch}),
one can  show the following:
\begin{lemma}\label{lmm: quasifbration q}
The map
$\hat{q}:\widehat{Q}^m_{n,\K}(B^*,\sigma B^*)\to
Q^m_{n,\K}(\overline{B},\sigma \overline{B})^2$
is a quasifibration with fiber
$\widehat{Q}^m_{n,\K}(B)$.
\end{lemma}
We postpone the proof of Lemma \ref{lmm: quasifbration q}
and complete the proof of Theorem \ref{thm: scanning map}.
\par\vspace{2mm}\par
For each $d\geq 1$, 
let $Q^{d,m}_{n,\K}\subset
Q^m_{n,\K}(B^*,\sigma B^*)$
denote the subspace of all
$m$-tuples
$(\xi_1,\cdots ,\xi_m)\in Q^m_{n,\K}(B^*,\sigma B^*)$ such that
$(\xi_1\cap B,\cdots ,\xi_m\cap B)\in Q^{d,m}_{n,\K}(B),$ 
and let
$q^d:Q^{d,m}_{n,\K}\to
Q^m_{n,\K}(B^*,\sigma B^*\cup \overline{B})\cong
Q^m_{n,\K}(\overline{B},\sigma \overline{B})^2$
be the quotient map.
While the map $q^d$ is not a fiber bundle, one can easily show see that each fiber of $q^d$
is homeomorphic to $Q^{d,m}_{n,\K}(B)$.
In fact, if we stabilized the map $q^d$, it will become a quasifibration.
\par
To see this,
define the map
$f^d:Q^{d,m}_{n,\K}\to Q^{d+2,m}_{n,\K}$ by
$(\xi_1,\cdots ,\xi_m)\mapsto
(\xi_1+c_{d,1}+\overline{c_{d,1}},\cdots ,
\xi_m+c_{d,m}+\overline{c_{d,m}})$, and
we denote by
$Q^{m}_{n,\K}$  the colimit 
$\dis \lim_{k\to\infty}Q^{1+2k,m}_{n,\K}$
taken from the maps $\{f^{1+2k}:k\geq 1\}$.
We also define another horizontal scanning map
$\hat{S}:
Q^{m}_{n,\K}\to 
\Map ([0,1],Q^m_{n,\K}(\overline{B},\sigma\overline{B}))$
by $\hat{S}(\xi_1,\cdots ,\xi_m)(t)=
(\xi_1\cap B_t,\cdots ,\xi_m\cap B_t)$,
where
$B_t$ denotes the open rectangle
$B_t=(2t-1,2t)\times (-1,1)$
and we use the canonical identification
$Q^m_{n,\K}(\overline{B}_t,\sigma \overline{B}_t)
\cong Q^m_{n,\K}(\overline{B},\sigma \overline{B}).$
\par

To see that $\hat{S}$ is a homotopy equivalence, consider the composite 
$$
Q^{m}_{n,\K}\to 
\Map ([0,1],Q^m_{n,\K}(\overline{B},\sigma\overline{B}))\to Q^m_{n,\K}(\overline{B},\sigma\overline{B}), $$
where the first map is $\hat{S}$ and the second evaluation at $\frac{1}{2}$. The composite is obviously a homotopy equivalence  hence so is $\hat{S}$.

Furthermore,
since $q^{d+2}\circ f^d=q^d$,
we obtain the stabilized map
$\dis q^{\infty}=\lim_{k\to\infty}q^{1+2k}:
Q^m_{n,\K}\to Q^m_{n,\K}(\overline{B},\sigma \overline{B})$
such that the following diagram is commutative
$$
\begin{CD}
Q^{m}_{n,\K}
@>q^{\infty}>> 
Q^m_{n,\K}(\overline{B},\sigma \overline{B})^2
\\
@V{\hat{S}}V{\simeq}V @VV{\cong}V
\\
\Map ([0,1],Q^{m}_{n,\K}(\overline{B},\sigma \overline{B}))
@>{res}>> 
\Map (\{0,1\},Q^{m}_{n,\K}(\overline{B},\sigma \overline{B}))
\end{CD}
$$
where $res$ denotes the restriction map.
Now consider the map
$g^d:Q^{d,m}_{n,\K}(B)\to Q^{d+2,m}_{n,\K}(B)$ given by
$(\xi_1,\cdots ,\xi_m)\mapsto
(\xi_1+c_{d,1}+\overline{c_{d,1}},\cdots ,
\xi_m+c_{d,m}+\overline{c_{d,m}})$.
Then in the same way as in Lemma \ref{lemma: stabilization map f}
we see that $Q^{\infty,m}_{n,\K}(B)$ can be identified 
with the colimit $\dis \lim_{k\to\infty}Q^{1+2k,m}_{n,\K}(B)$ taken from the maps
$\{g^{1+2k}:k\geq 1\}$.
Note that the space $\widehat{Q}^m_{n,\K}(B^*,\sigma B^*)$ has $\Z^m$-path components. An element
$\eta =(\xi_1,\cdots ,\xi_m)\in \widehat{Q}^m_{n,\K}(B^*,\sigma B^*)$
belongs to the $(n_1,\cdots ,n_m)$-th component iff 
$\deg (\xi_k\cap (B^*\setminus \sigma B^*)-\xi_k^*)=n_k$ for each 
$1\leq k\leq m$.
Since we can identify $\widehat{Q}^m_{n,\K}(B^*,\sigma B^*)$ with
$\Z^m\times Q^m_{n,\K}$, we can see that
$\hat{q}\vert Q^m_{n,\K}=q^{\infty}$ and that
$Q^{\infty,m}_{n,\K}=\Z^m\times Q^{\infty,m}_{n,\K}(B).$
Moreover, by using this identification, we see that
$\hat{S}\vert Q^{\infty,m}_{n,\K}(B)=S^{\p}$.
Thus, we finally obtain the quasifibration sequence
$$
Q^{\infty,m}_{n,\K}(B)\longrightarrow
Q^m_{n,\K} \stackrel{q^{\infty}}{\longrightarrow}
Q^m_{n,\K}(\overline{B},\sigma \overline{B})^2
$$
such that the following diagram is commutative:
$$
\begin{CD}
Q^{\infty,m}_{n,\K}(B) 
@>>>
Q^{m}_{n,\K}
@>q^{\infty}>> 
Q^m_{n,\K}(\overline{B},\sigma \overline{B})^2
\\
@V{S^{\p}}VV @V{\hat{S}}V{\simeq}V @VV{\cong}V
\\
\Omega Q^m_{n,\K}(\overline{B},\sigma \overline{B})
@>>>
\Map ([0,1],Q^{m}_{n,\K}(\overline{B},\sigma \overline{B}))
@>{res}>> 
\Map (\{0,1\},Q^{m}_{n,\K}(\overline{B},\sigma \overline{B}))
\end{CD}
$$
where
 the lower sequence is a fibration sequence.
Thus $S^{\p}$ is a homotopy equivalence.
\end{proof}
%
\begin{proof}[Proof of Lemma \ref{lmm: quasifbration q}]
Let us consider the family of closed subspaces
$\{Q_{d_1,\cdots ,d_m}:d_j\geq 1 \mbox{ for } 1\leq j\leq m\}$
of the base space $Q^m_{n,\K}(B^*,\sigma B^*\cup \overline{B})$
defined by
$$
Q_{d_1,\cdots ,d_m}=\{(\xi_1,\cdots ,\xi_m)\in 
Q^m_{n,\K}(B^*,\sigma B^*\cup \overline{B}):
\deg (\xi_j\cap (B^*\setminus B^{\p})\leq d_j
\mbox{ for each }j\},
$$
where we write $B^{\p}:=\sigma B^*\cup \overline{B}.$
\par
Suppose that $\hat{q}$ is a quasifibration over the subspace
$Q^{\p}=
Q_{d_1,\cdots ,d_m}.$
Using the Dold-Thom criterion
\cite[Lemma 3.3]{McDuff}
(cf. \cite[Lemma 4K.3]{Hatch}),
it suffices show that $\hat{q}$ is a quasifibration over
$Q=Q_{d_1,\cdots ,d_{k-1},d_k+2,d_{k+1},\cdots ,d_m}$.
If we write $Q^{\p\p}=Q\setminus Q^{\p}$, it is easy to see that
$\hat{q}$ is a trivial bundle over $Q^{\p\p}$.
Let $U$ be an open neighborhood of $\hat{q}^{-1}(Q^{\p})$ in
$\hat{q}^{-1}(Q)$ given by $U=\hat{q}^{-1}(Q^{\p})\cup V$, where 
$0<\delta<10^{-10}$ denotes any fixed small number
with
$I_{\delta}=\big((-\delta,0)\cup (1+\delta)\big)\times (0,1),$
and $V$ denotes the open subspace of $\hat{q}^{-1}(Q^{\p\p})$
consisting of all $m$-tuples
$(\xi_1,\cdots ,\xi_{k-1},\xi_k+x_0+\overline{x_0},
\xi_{k+1},\cdots ,\xi_m)$ with $x_0\in I_{\delta}$
and $(\xi_1,\cdots ,\xi_m)\in \hat{q}^{-1}(Q^{\p})$, such that
$\min\{|x_0-x|,|\overline{x_0}-x|:x\in \xi_i
\mbox{ for }1\leq i\leq m\}\geq 100\delta$.
By moving the points $x_0$ and $\overline{x_0}$ continuously to a point of $\sigma B$
by suitable "gravitational" force, we obtain the homotopies
$
H_t:\hat{q}^{-1}(Q)\to \hat{q}^{-1}(Q)
$ 
and
$
h_t:Q\to Q$ 
such that
$H_0=\mbox{id}$, $h_0=\mbox{id}$, 
$H_1(\hat{q}^{-1}(Q))\subset \hat{q}^{-1}(Q^{\p})$
and
$\hat{q}\circ H_t=h_t\circ \hat{q}$
for each $0\leq t\leq 1$. 
\par
It remains to show that the Dold-Thom attaching map $H_1:\hat{q}^{-1}(x)\to \hat{q}^{-1}(h_1(x))$ is a homotopy equivalence
for any $x\in Q.$
If we identify these fibers of $\hat{q}$ with $\hat{Q}^m_{n,\K}(B)$,
this attaching map can be regarded as the map
$\hat{Q}^m_{n,\K}(B)\to \hat{Q}^{m}_{n,\K}(B)$, and
 it is given by
$(\eta_1,\cdots ,\eta_m)\mapsto
(\eta_1,\cdots ,\eta_{k-1},\eta_k+x_0+\overline{x_0},\eta_{k+1},\cdots ,\eta_m).$
If we identify $\hat{Q}^m_{n,\K}(B)= \Z^m\times
Q^{\infty,m}_{n}(B)$, this map is   
the shift map  
on the component $(n_1,\cdots ,n_m)\mapsto
(n_1,\cdots ,n_{k-1},n_k+2,n_{k+1},\cdots ,n_m).$
Thus it induces a homology isomorphism for any local coefficients.
Since 
$\pi_1(Q^{\infty,m}_{n}(B))$ is an abelian group
by  Lemma \ref{lmm: abelian},
we see that this attaching map is a homotopy equivalence.
\end{proof}


\begin{definition}\label{def: 5.9}
(i)
Let $\mathcal{P}^d(\K)$ denote the space of all
(not necessarily monic) 
polynomials $f(z)=\sum_{i=0}^da_iz^i\in\K [z]$ of degree exactly $d$ 
and let  $\mathcal{P}oly^{d,m}_{n,\K}$
denote the space of all
$m$-tuples $(f_1(z),\cdots ,f_m(z))\in \mathcal{P}^d(\K )^m$
such that polynomials $\{f_1(z),\cdots ,f_m(z)\}$ have no common {\it real} root of
multiplicity $\geq n$.
\par
(ii)
For each nonempty subset $X\subset \C$, let
$\pol^{m}_{n,\K}(X)$
denote the space of
all $m$-tuples $(f_1(z),\cdots ,f_m(z))\in \K [z]^m$ of polynomials of the same degree such that polynomials $\{f_1(z),\cdots ,f_m(z)\}$ have no common real root in $X$ of multiplicity
$\geq n$,  and let $\mathcal{P}oly^{m}_{n,\K}(X)\subset \pol^{m}_{n,\K}(X)$ denote the subspace of $(f_1(z),\cdots ,f_m(z))\in \pol^{m}_{n,\K}(X)$ such that
no $f_i(z)$ is identically zero.
When $X=\C$, we write
\begin{equation} 
\mathcal{P}oly^{m}_{n,\K}=\mathcal{P}oly^{m}_{n,\K}(\C).
\end{equation}
\end{definition}

\begin{remark}
It is easy to see that there are homeomorphisms
$$
\mathcal{P}^d(\K)\cong \K^*\times\P^d(\K)
\ \mbox{ and }\ 
\mathcal{P}oly^{d,m}_{n,\K}\cong
\T^m_{\K} \times \Q^{d,m}_{n}(\K),
$$
where we set $\K^*=\K\setminus \{0 \}$ and 
$\T^m_{\K}=(\K^*)^m$.
\qed
\end{remark}
\par\vspace{2mm}\par

Next we consider horizontal scanning for algebraic maps.

\begin{definition}
(i)
We identify $\C=\R^2$ in a usual way.
Let
\begin{equation}
U=\{w\in\C :\vert \mbox{Re}(w)\vert <1, \ \vert\mbox{Im}(w)\vert <1\}=(-1,1)\times (-1,1)
\end{equation}
and 
define {\it the horizontal scanning map}
for  $\mathcal{P}oly^{d,m}_{n,\K}$,
\begin{align}\label{equ: scan2}
\mbox{sca}^{d,m}_n:&
\mathcal{P}oly^{d,m}_{n,\K}
\to \Map(\R, \mathcal{P}oly^{m}_{n,\K}(U))
\qquad
\mbox{by}
\\
\nonumber
\mbox{sca}^{d,m}_n&(f_1(z),\cdots ,f_m(z))(x)=
(f_1\vert V(x),\cdots ,f_m\vert V(x))
\end{align}
for $((f_1(z),\cdots ,f_m(z),x)\in \mathcal{P}oly^{d,m}_{n,\K}\times \R$,
where  $V(x)$ denotes the space given by
(\ref{eq: Vx}) and
 we also use the canonical
identification $U\cong V(x)$.
\par
(ii)
Let us identify $\overline{U}=I^2$ and let
$
q:\mathcal{P}oly^m_{n,\K}(U)\to 
Q^{m}_{n,\K}(I^2,\partial I\times I)
$
denote the map given by assigning to an $m$-tuples of polynomials their
corresponding roots in $U$. 
\qed
\end{definition}

\begin{lemma}\label{lemma: fiber}
Any fiber of the map $q$ is homotopy equivalent to the space
$\T^m_{\K}$.
\end{lemma}
\begin{proof}
Any fiber of the map $q$ is homeomorphic to the space
$F(m)$ consisting of all $m$-tuples
$(f_1(z),\cdots ,f_m(z))\in \K[z]^m$ of $\K$-coefficients polynomials such that
each polynomial $f_i(z)$ has no root in $U$.
It suffices to show that there is a homotopy equivalence
$F(m)\simeq \T^m_{\K}$.
First define the inclusion map $j_0:\T^m_{\K}\to F(m)$ by
$j_0(\textit{\textbf{x}})=(x_1,\cdots ,x_m)$
for $\textit{\textbf{x}}=(x_1,\cdots ,x_m)\in \T^m_{\K}$.
Next,
let $f=(f_1(z),\cdots ,f_m(z))\in F(m)$ be any element.
Since $0\in U$,
$(f_1(0),\cdots ,f_m(0))\in \T^m_{\K}$.
Hence, one can define the evaluation map
$e_0:F(m)\to \T^m_{\K}$
by
$e_0(f)=(f_1(0),\cdots ,f_m(0))$ for
$f=(f_1(z),\cdots ,f_m(z))\in F(m)$.
It is easy to see that $e_0\circ j_0=\mbox{id}_{\T^m_{\K}}.$
\par
Now consider the map $j_0\circ e_0$.
Note that if a polynomial $g(z)\in \K [z]$ has a root $\alpha \in \C\setminus U$ and
$0<t\leq 1$, the polynomial $g(tz)$ has a root $\alpha/t\in \C\setminus U$.
Thus, one can define the homotopy
$F:F(m)\times [0,1]\to F(m)$ by
$F(f,t)=(f_1(tz),\cdots ,f_m(tz))$
for $(f,t)=((f_1(z),\cdots ,f_m(z)),t)\in F(m)\times [0,1].$
It is easy to see that the map $F$ gives a homotopy between
the maps
$j_0\circ e_0$ and $\mbox{id}_{F(m)}$.
Hence, we see that the map $e_0:F(m)\stackrel{\simeq}{\longrightarrow}\T^m_{\K}$
is a homotopy equivalence.
\end{proof}

\begin{lemma}\label{lmm: quasi-fibration}
The map 
$q:\mathcal{P}oly^m_{n,\K}(U)\to 
Q^{m}_{n,\K}(I^2,\partial I\times I)
=Q^{m}_{n,\K}(\overline{U},\sigma \overline{U})$ 
is a quasifibration with fibre $\T^m_{\K}$.
\end{lemma}
\begin{proof}
As before we identify $\C =\R^2$.
The assertion may be proved by using the well-known criterion of Dold-Thom.
The basic idea, which was first applied to the study of spaces pairs of polynomials 
 without common roots, is due to \cite[Lemma 3.3]{Se}.  Our argument is almost identical to the one given for the case $m=1$ and complex polynomials  without $n$-fold roots in \cite{GKY2}. As this argument is somewhat easier to grasp, we begin by briefly repeating it. 

For a non-empty open subset $X\subset \Bbb C$,  let $SP_n(X)$ denote the space of all complex (not necessarily monic) polynomial functions $f(z)=\sum a_i z^i$ such that every root of $f(z)$ in $X$ has multiplicity less than $n$.
The set $SP_n(X)$ is topologized as the subspace of the space of all polynomials (which can be identified with the space of all elements of the infinite cartesian product $\Bbb C^\infty$ with only finitely many non-zero terms).
 Note that $SP_n(\Bbb C)$, although bijectively equivalent to the disjoint union 
$\coprod _{d\ge 0} SP_n^d(\Bbb C)$ of the spaces 
$SP^d_n(\C)$ of complex polynomials of degree exactly $d$,  is connected. For example the degree one polynomial  $a z - b$ can be connected to the degree zero polynomial  $-b$ by the homotopy $t a z-b$, where $t \in [0,1]$. 
\par
Let $V=\{x\in \C: |x|<1\}$ and let
$\text{SP}_n(\overline{V},\partial \overline{V})$ denote the space of all configurations of points in $\overline{V}$ without multiplicity larger than $n$, with two configurations identified if they differ only on the boundary (we can also think that points can enter the boundary and vanish).
\par
Consider the natural projection map $q: SP_n(V) \to \SP_n(\overline{V},\partial \overline{V})$.
In \cite{GKY2} this map is shown to be a quasifibration as follows. We filter the base space $\SP_n(\bar V,\partial \bar{V})$ by the number of  points in $V$. It is easy to see that over the each successive difference in this filtration the map $q$ is a locally trivial fibre bundle. We can now use the Dold-Thom method 
(see Lemma 3.3 and the proof of 
\cite[Proposition 3.2]{Se}) to show that $q$ is a quasifibration. The Dold-Thom attaching map 
on the fibre has the effect of multiplying polynomials with no roots in $V$ by a fixed polynomial $(z-\alpha)$, where $\alpha$ lies outside $V$. Since $\alpha$ can be moved continuously to $1$, the corresponding map on the fibre $\Bbb C^*$ is a homotopy equivalence. 

In the case of polynomials with real coefficients the argument is identical except that, since  non-real roots of real polynomials must occur in conjugate pairs, the Dold-Thom attaching map is this time multiplication by $(z-a) (z-\bar a)$, for some complex $a\notin V$. This map, of course, is also homotopic to the identity. 
\par\vspace{1mm}\par
Next, turning to the case of $m$-tuples of configurations $Q^{m}_{n,\K}(I^2,\partial I\times I)$, we use the essentially the same argument for proving Lemma \ref{lmm: quasi-fibration}. 
\par
The difference is that we are now dealing with $m$-tuples of configurations. To avoid complexities of notation we shall only consider the case $m=2$, and we write
$B=Q^{2}_{n,\K}(I^2,\partial I\times I).$
The argument is essentially given in \cite[Lemma 3.3]{Se}, but as it is only a sketch, we will give a more detailed argument. 
\par
Recall that
 the base space $B$ consists of pairs of divisors (or configurations)  $(\xi_1,\xi_2)$ without common real $n$-fold points in $\overline{U}\setminus \sigma \overline{U}$. We filter the base space $B$ by an increasing family of subspaces 
 $\{B_{\le d_1,\le d_2}\}$, where $B_{\le d_1,\le d_2}$ 
 denotes the space  of pairs $(\xi_1,\xi_2)\in B$ such that 
$\deg (\xi_i\cap \mathbb R \cap (\overline{U}\setminus \sigma \overline{U}))\le d_i$
for each $i=1,2$.
\par
We aim to prove that $q$ restricted to $B_{\le d_1,\le d_2}$ is a quasifibration
for each pair $(d_1,d_2)$ of posiitve integers. 
For each $0\le d\le d_1$, let $B_{d,\le d_2}$ denote the subspace of $B_{\le d_1,\le d_2}$ 
consisting of pairs $(\xi_1,\xi_2)$ for which
 $\deg (\xi_1 \cap \mathbb R \cap (\overline{U}\setminus \sigma \overline{U}))=d$. 
We want to prove that  the map $q$ restricted to this space is a quasifibration. 
We filter this space by subspaces $\{B_{d,\le d^{\p}}:0\le d^{\p} \le d_2\}$.  
Over each difference 
$B_{d,d^{\p}} = B_{d,\le d^{\p}} \setminus  B_{d,\le d^{\p}-1}$ the map $q$ is a locally trivial fibre bundle. We now use the Dold-Thom technique, exactly as in  
\cite[Lemma 3.3]{Se} to stitch these quasifibrations, to one over $B_{d,\le d_2}$. 
We do this for each $d\le d_1$ and then turn to $B_{\le d_1,\le d_2}$. 
We now filter it by subspaces  $\{B_{\le d,\le d_2}:d\le d_1\}$, where $d_2$ is now fixed. The differences are now precisely the spaces $B_{d,\le d_2}$, over which we have already proved that the map $q$ is a quasifibration. Repeating the Dold-Thom argument concludes the proof. 
The process in which we stitch a fibre bundle defined over an open subset and a quasifibration defined over a closed one, requires finding an open set containing the closed one and a suitable retraction that lifts to the total space and induces an isomorphism on fibres, exactly as in the proof of  Lemma \ref{lmm: quasifbration q}. 
\end{proof}

\begin{definition}
(i)
Let
$ev:\pol^{m}_{n,\K}(U)\to \K^{mn}\setminus \{{\bf 0}\}$
denote the evaluation map at $z=0$ given by
\begin{eqnarray}\label{eqna: ev}
ev(f_1(z),\cdots ,f_m(z))
&=&
(\textit{\textbf{f}}_1(0),\textit{\textbf{f}}_2(0)\cdots ,
\textit{\textbf{f}}_m(0))
\end{eqnarray}
for $(f_1(z),\cdots ,f_m(z))\in \pol^m_{n,\K}(U)$, and let
$ev_0:\mathcal{P}oly^{m}_{n,\K}(U)\to \K^{mn}\setminus \{{\bf 0}\}$ denote
the restriction $ev_0=ev\vert \mathcal{P}oly^{m}_{n,\K}(U).$
\par
(ii) Let $G$ be a group and $X$ a $G$-space.
Then we denote by $X//G$ the homotopy quotient of $X$ by $G$,
$X//G=EG\times_{G}X$, where
$EG$ denotes the contractible free $G$-space.
\end{definition}

\begin{remark}
Let $\T^m_{\K}=(\K^*)^m$ and
 consider the diagonal $\T^m_{\K}$-action on the spaces
$\mathcal{P}oly^m_{n,\K}(U)$ and $\K^{mn}\setminus \{{\bf 0}\}$ 
given by
$$
\begin{cases}
(g_1,\cdots ,g_m)\cdot (f_1(z),\cdots ,f_m(z))
&=
(g_1f_1(z),\cdots ,g_mf_m(z))
\\
(g_1,\cdots ,g_m)\cdot
(\textit{\textbf{x}}_1,\cdots ,\textit{\textbf{x}}_m)
&=
(g_1\textit{\textbf{x}}_1,\cdots ,g_m\textit{\textbf{x}}_m)
\end{cases}
$$
for $(g_1,\cdots ,g_m))\in \T^m_{\K}$ and
$\{\textbf{\textit{x}}_i\}_{i=1}^m\subset \K^n$.
Note that $ev_0$ is a $\T^m_{\K}$-equivariant map.
\end{remark}

\begin{lemma}\label{lmm: evaluation0}
The map
$ev_0:\mathcal{P}oly^m_{n,\K}(U)\stackrel{\simeq}{\longrightarrow}
 \K^{mn}\setminus \{{\bf 0}\}$
is a  homotopy equivalence.
\end{lemma}
\begin{proof}
For each $\textbf{\textit{x}}=(x_0,\cdots ,x_{n-1})\in \K^n$,
let $\varphi_{\textbf{\textit{x}}}(z)\in \K [z]$ denote the
polynomial with coefficients in $\K$ defined by
$\varphi_{\textbf{\textit{x}}}(z)
=x_0+\sum_{k=1}^{n-1}\big(\frac{x_k-x_0}{k!}\big)z^k.$
Since the degree of the polynomial 
$\varphi_{\textbf{\textit{x}}}(z)$ is at most $n-1$, 
it has no root of multiplicity $\geq n$.
Thus
one can define the natural inclusion map $i_0:\K^{mn}\setminus \{{\bf 0}\}\to
\pol^{m}_{n,\K}(U)$ by
$$
i_0(\textit{\textbf{x}}_1,\cdots ,\textit{\textbf{x}}_m)
=(\varphi_{\textbf{\textit{x}}_1}(z),\cdots ,
\varphi_{\textbf{\textit{x}}_m}(z))
\qquad
\mbox{for }\ 
(\textit{\textbf{x}}_1,\cdots ,\textit{\textbf{x}}_m)\in\K^{mn}\setminus
\{{\bf 0}\}.
$$
It is easy to see that $ev\circ i_0=\mbox{id}.$
Next consider the homotopy
$F:\pol^m_{n,\K}(U)\times [0,1]\to \pol^m_{n,\K}(U)$ defined by
$F((f_1(z),\cdots ,f_m(z)),t)=(f_1(tz),\cdots ,f_m(tz)).$
Then $F_1=F(\ ,1)=\mbox{id}$
and the map $F_0=F(\ ,0)$ is given by
$
F_0(f_1(z),\cdots ,f_m(z))=
(f_1(0),\cdots ,f_m(0)).
$
Note that for $f=(f_1(z),\cdots ,f_m(z))\in \pol^m_{n,\K}(U)$
$$
(i_0\circ ev)(f)
=
\Big(f_1(0)+\sum_{k=1}^{n-1}\frac{f^{(k)}_1(0)}{k!}z^k,
\cdots ,f_m(0)+\sum_{k=1}^{n-1}\frac{f^{(k)}_m(0)}{k!}z^k\Big).
$$
Define the homotopy
$G:\pol^m_{n,\K}(U)\times [0,1]\to \pol^m_{n,\K}(U)$ by
$$
G((f_1(z),\cdots ,f_m(z)),t)
=
\Big(f_1(0)+t\sum_{k=1}^{n-1}\frac{f^{(k)}_1(0)}{k!}z^k,
\cdots ,f_m(0)+t\sum_{k=1}^{n-1}\frac{f^{(k)}_m(0)}{k!}z^k\Big).
$$
Since this gives the homotopy between the maps
$F_0$ and $i_0\circ ev$,
the map $i_0\circ ev$ is homotopic to the identity map.
Thus, the map $ev$ is a homotopy equivalence. 
\par
Note that the complement
$\Sigma =\pol^m_{n,\K}(U)\setminus \mathcal{P}oly^m_{n,\K}(U)$
is the space of all $m$-tuples
$(f_1(z),\cdots ,f_m(z))\in \pol^m_{n,\K}(U)$ such that some
$f_i(z)$ is identically zero.
Note that $\pol^m_{n,\K}(U)$ is an infinite dimensional manifold and
that
$\Sigma$ is a finite union of linear subspaces, each of infinite codimension.
Then by using the induction on the number of  linear subspaces and by
\cite[Theorem 2]{EK}, we can show that
the inclusion $\mathcal{P}oly^m_{n,\K}(U)\to \pol^m_{n,\K}(U)$ is a
homotopy equivalence.
Thus the map $ev_0=ev\vert  \mathcal{P}oly^m_{n,\K}(U)$ is also a homotopy equivalence.
%
\end{proof}

\begin{definition}
Now recall the map
$j^{d,m}_{n,\K}:\Q^{d,m}_n(\K)\to \Omega_{[d]_2}\RP^{d(\K)mn-1}\simeq
\Omega S^{d(\K)mn-1}$ and the stabilized space
$\dis \Q^{\infty,m}_n(\K)=\lim_{d\to\infty}\Q^{d,m}_n(\K)$ given in
Definition \ref{def: stabilization}.
Since there is a homotopy commutative diagram
\begin{equation}
\begin{CD}
\Q^{d,m}_n(\K) @>j^{d,m}_{n,\K}>> \Omega S^{d(\K)mn-1}
\\
@V{s^{d,m}_{n,\K}}VV \Vert @.
\\
\Q^{d+1,m}_n(\K) @>j^{d+1,m}_{n,\K}>> \Omega S^{d(k)mn-1}
\end{CD}
\end{equation}
for each $d\geq 1$, these maps induces the following map
\begin{equation}\label{equ: natural map}
j^{\infty,m}_{n,\K}
:\Q^{\infty,m}_{n}(\K)
\to
\Omega S^{d(\K)mn-1}.
\end{equation}
\end{definition}

\begin{theorem}\label{thm: natural map}
If $d(\K)mn\geq 4$, the map
$
j^{\infty,m}_{n,\K}:\Q^{\infty, m}_{n}(\K)
\stackrel{\simeq}{\longrightarrow}
\Omega S^{d(\K)mn-1}
$
is a homotopy equivalence.
\end{theorem}


\begin{proof}
Note that the group $\T^m_{\K}$ acts on
$\mathcal{P}oly^m_{n,\K}(U)$ freely, but it
does not act on 
$\K^{mn}\setminus \{{\bf 0}\}$  freely. So we have to consider the homotopy quotient $(\K^{mn}\setminus \{{\bf 0}\})//\T^m_{\K}$ of the action.
Since $ev_0$ is a $\T^m_{\K}$-equivariant map, we obtain the following commutative diagram
$$
\begin{CD}
\T^m_{\K} @>>> \mathcal{P}oly^m_{n,\K}(U)
@>q_1>> \mathcal{P}oly^m_{n,\K}(U)/\T^m_{\K}
\\
\Vert @. @V{ev_0}V{\simeq}V @V{\widetilde{ev}_0}VV
\\
\T^m_{\K}  @>>> \K^{mn}\setminus\{{\bf 0}\}
@>q_2>> (\K^{mn}\setminus \{{\bf 0}\})//\T^m_{\K}
\end{CD}
$$
where two horizontal sequences are fibration sequences and
and each $q_i$ $(i=1,2)$  is natural projection induced from the group action.
Thus, we  see that $\widetilde{ev}_0$ is a homotopy equivalence.
If
 $q^{\p}:
 \mathcal{P}oly^{d,m}_{n,\K}/\T^m_{\K} 
 \to Q^m_{n,\K}(I^2,\partial I\times I)$ 
 denotes the induced map from the map $q$, then
it
is a  homotopy equivalence
by  Lemma \ref{lmm: quasi-fibration}.
Since $\dis\lim_{t\to\pm \infty}\mbox{sca}^{d,m}_n(f)(t)=
(\emptyset,\cdots ,\emptyset)$ for any $f\in Poly^{d,m}_{n,\K}$,
the map $\mbox{sca}^{d,m}_n$ can be extended to the based map
$\mbox{sca}^{d,m}_n:Poly^{d,m}_{n,\K}\to\Omega Poly^{d,m}_{n,\K}(U)$
by 
$\infty\mapsto \mbox{the constant loop at }(\emptyset,\cdots ,\emptyset )$,
where
we identify $S^1=\R\cup \infty$ and we choose the points $\infty$ and 
$(\emptyset,\cdots ,\emptyset)$ as base-points of $S^1$ and
$\Omega Poly^{d,m}_{n,\K}(U)$, respectively.
Now consider the following commutative diagram
{\small
$$
\begin{CD}
\mathcal{P}oly^{d,m}_{n,\K} 
@>\mbox{sca}^{d,m}_n>>
\Omega \mathcal{P}oly^m_{n,\K}(U)
@>\Omega ev_0>\simeq> 
\Omega (\K^{mn}\setminus\{{\bf 0}\})
\\
@V{q_3}VV @V{\Omega q_1}VV @V{\Omega q_2}VV
\\
\mathcal{P}oly^{d,m}_{n,\K}/\T^m_{\K} 
@>>> 
\Omega
(\mathcal{P}oly^m_{n,\K}(U)/\T^m_{\K})
@>\Omega \widetilde{ev}_0>\simeq> 
\Omega (\K^{mn}\setminus \{{\bf 0}\})//\T^{m}_{\K})
\\
@V{\cong}VV @V{q^{\p}}V{\simeq}V @.
\\
\Q^{d,m}_{n}(\K) @>sc^{d,m}_n>> 
\Omega 
Q^m_{n,\K}(I^2,\sigma I^2)
@.
\end{CD}
$$
}
\newline
where %
the map $q_3$ 
is induced from 
the corresponding group action.
Now consider the map $\gamma_{\K}$ given by the second row of the above diagram.
Since $sc^{d,m}_n$ is a homotopy equivalence if $d\to \infty$
(by Theorem \ref{thm: scanning map}), the map $\gamma_{\K}$
is a homotopy equivalence if $d\to\infty$.
\par
First, consider the case $\K =\R$.
Since $\T^m_{\R}=\{\pm 1\}^m$,
the two  maps $\Omega q_i$ $(i=1,2)$ are homotopy equivalences.
However, since
the map $\gamma_{\R}$ coincides the map $j^{d,m}_{n,\R}$
(if $d\to \infty$)
up to homotopy equivalence, the map
$j^{\infty,m}_{n,\R}$ is a homotopy equivalence.
\par
Next consider the case $\K =\C$.
Since $\Q^{d,m}_n(\C)$ is  simply connected by Lemma \ref{lmm: 1-connected},
the space $\Q^{\infty,m}_n(\C)$ is also simply connected.
Since $\Omega S^{2mn-1}$ is simply connected,
it suffices to prove that $j^{\infty,m}_{n,\C}$ induces an isomorphism on homotopy groups $\pi_k(\ )$ for any $k\geq 2$.
\par
Since $\T^m_{\C}\simeq (S^1)^m$, we see that
the two maps $\Omega q_i$ $(i=1,2)$ induce isomorphisms on homotopy groups
$\pi_k(\ )$ for any $k\geq 2$.
Since $\gamma_{\C}$ is a homotopy equivalence if $d\to\infty$,
the map $j^{\infty,m}_{n,\C}$ also induces an isomorphism on homotopy groups
$\pi_k(\ )$ for any $k\geq 2$.
\end{proof}
\section{Proof of the main results}\label{section: proofs}

In this section, we give the proofs of the main results
(Theorem \ref{thm: I}, Corollary \ref{cor: I-2} and
Theorem \ref{thm: II}).


\begin{proof}[Proof of Theorem \ref{thm: I}]
Suppose that $d(\K)mn\geq 4$.
Note that the two spaces $\Q^{d,m}_{n}(\K)$ and $\Omega S^{d(\K)mn-1}$ are simply
connected 
(by Lemma \ref{lmm: 1-connected}).
Hence, 
the assertion follows from Theorem \ref{thm: stab1} and
Theorem \ref{thm: natural map}.
%
\end{proof}

\begin{proof}[Proof of Corollary \ref{cor: I-2}]
It is easy to see that the following diagram is commutative:
$$
\begin{CD}
\Q^{d,m}_{n}(\K)
@>j^{d,m}_{n,\K}>> \Omega S^{d(\K)mn-1}%
\\
@V{i^{d,m}_{n,\K}}VV \Vert @.
\\
\Q^{d,mn}_{1}(\K) @>j^{d,mn}_{1,\K}>>  \Omega S^{d(\K)mn-1} 
\end{CD}
$$
By Theorem \ref{thm: I}
the two maps
$j^{d,m}_{n,\K}$ and $j^{d,mn}_{1,\K}$ are
homotopy equivalences through dimension
$D(d;m,n,\K)$ and $D(d;mn,1,\K)$, respectively.
Since $D(d;m,n,\K)<D(d;mn,1,\K)$, it follows from the above commutative diagram that
the map
$i^{d,m}_{n,\K}$ is a homotopy equivalence through dimension $D(d;m,n,\K)$.
\end{proof}

\begin{proof}[Proof of Theorem \ref{thm: II}]
Suppose that $d(\K)mn\geq 4$.
It suffices to prove the first assertion.
Since $\Q^{d,m}_{n}(\K)$ is simply connected (by Lemma \ref{lmm: 1-connected}),
it follows from Corollary \ref{crl: homology} and the cellular approximation theorem that
there is a map
$
f:\Q^{d,m}_{n}(\K)\to J_{\lfloor\frac{d}{n}\rfloor}(\Omega S^{d(k)mn-1})
$
such that the following diagram is homotopy commutative:
$$
\begin{CD}
\Q^{d,m}_{n}(\K) @>j^{d,m}_{n,\K}>> \Omega S^{d(\K)mn-1}
\\
@V{f}VV \Vert @.
\\
J_{\lfloor\frac{d}{n}\rfloor}(\Omega S^{d(k)mn-1})
@>i>\subset> \Omega S^{d(\K)mn-1}
\end{CD}
$$
where $i$ is the natural inclusion map.
Since the maps $i^{d,m}_{n,\K}$ and  $i$ are homotopy equivalences through dimension
$D(d;m,n,\K)$, 
the map $f$ is a homotopy equivalence through dimension $D(d;m,n,\K)$, too.
Note that $H_k(\Q^{d,m}_{n,\K};\Z)=
H_k(J_{\lfloor\frac{d}{n}\rfloor}(\Omega S^{d(k)mn-1});\Z)=0$ for any
$k>D(d;m,n,\K)$. Hence, the map $f$ is a homology equivalence.
Since the spaces $\Q^{d,m}_{n}(\K)$ and $J_{\lfloor\frac{d}{n}\rfloor}(\Omega S^{d(k)mn-1})$
are simply connected, the map $f$ is indeed a homotopy equivalence.
\end{proof}
\paragraph{Acknowledgements.}
The authors are very grateful to the referee for his vey useful suggestions and 
comments.
The second author was supported by 
JSPS KAKENHI Grant Number 26400083 and 18K03295.
This work was also supported by the Research Institute for Mathematical Sciences, a Joint Usage/Research Center located in Kyoto University.

%
%
%
%
%
%


\end{document}